\documentclass[a4paper,10pt,english,reqno]{amsart}

\usepackage{color}
\usepackage{dsfont}
\usepackage{hyperref}
\hypersetup{colorlinks=false}

\newcommand{\marginlabel}[1]{\mbox{}\marginpar{\raggedleft\hspace{0pt}\tiny{\textcolor{red}{#1}}}}
\renewcommand{\marginlabel}[1]{} 

\newcommand{\R}{\mathbb{R}}
\newcommand{\N}{\mathbb{N}}
\newcommand{\Z}{\mathbb{Z}}
\newcommand{\abs}[1]{\left|#1\right|}
\newcommand{\abseps}[1]{{\abs{#1}_\eps}}
\newcommand{\eps}{\varepsilon}
\newcommand{\seq}[1]{\left\{#1\right\}}
\newcommand{\norm}[1]{\left\Vert #1\right\Vert}
\newcommand{\sgn}[1]{\mathrm{sign}\left(#1\right)}

\newcommand{\Dt}{{\Delta t}}

\newcommand{\test}{\varphi}
\newcommand{\Dx}{{\Delta x}}

\newcommand{\Dm}{D_-}
\newcommand{\Dp}{D_+}

\newcommand{\Dpt}{D_+^t}
\newcommand{\Dmt}{D_-^t}
\newcommand{\loc}{{\mathrm{loc}}}
\newcommand{\car}[1]{\mathds{1}_{\seq{#1}}}
\newcommand{\downto}{\downarrow}

\newcommand{\taum}{\tau^-}
\newcommand{\taup}{\tau^+}
\newcommand{\thetap}{\theta^+}
\newcommand{\thetam}{\theta^-}

\newcommand{\zetae}{\zeta^\varepsilon}
\newcommand{\Ee}{E^\varepsilon}
\newcommand{\sg}[2]{\mathrm{sg}_{#1}\left(#2\right)}

\newcommand{\signe}{\textnormal{sign}_\varepsilon}
\newcommand{\sign}{\textnormal{sign}}
\newcommand{\interval}{\mathrm{int}}
\newcommand{\psie}{\psi_\varepsilon}

\newcommand{\intPi}{\iint_{\Pi_T^2}}

\newcommand{\qe}{q_\varepsilon}

\newcommand{\dxdt}{\,dxdt}
\newcommand{\dyds}{\,dyds}
\newcommand{\dX}{\,dX}
\newcommand{\dz}{\,dz}

\theoremstyle{plain} \newtheorem{lemma}{Lemma}[section]
\theoremstyle{definition} \newtheorem{example}{Example}[section]
\theoremstyle{definition} \newtheorem{definition}{Definition}[section]
\theoremstyle{definition} \newtheorem{estimate}{Estimate}[section]
\theoremstyle{plain} \newtheorem{theorem}{Theorem}[section]
\theoremstyle{definition} \newtheorem{remark}{Remark}[section]

\numberwithin{equation}{section}
\allowdisplaybreaks[1]

\title[Error estimates for convection-diffusion equations]
{$L^1$ error estimates for difference approximations 
of degenerate convection-diffusion equations}

\author[K. H. Karlsen]{K. H. Karlsen} \address[Kenneth H. Karlsen]
{\newline Center of Mathematics for Applications (CMA) 
\newline University of Oslo
\newline P.O. Box 1053,  Blindern
\newline N--0316 Oslo, Norway} 
\email[]{kennethk@math.uio.no}

\author[N. H. Risebro]{N. H. Risebro} \address[Nils Henrik Risebro]
{\newline Center of Mathematics for Applications (CMA)
\newline University of Oslo
\newline P.O. Box 1053, Blindern
\newline N--0316 Oslo, Norway}
\email[]{nilshr@math.uio.no}

\author[E. B. Storr\o{}sten]{E. B. Storr\o{}sten} \address[Erlend Briseid Storr\o{}sten]
{\newline Center of Mathematics for Applications (CMA)
\newline University of Oslo
\newline P.O. Box 1053, Blindern
\newline N--0316 Oslo, Norway} 
\email[]{erlenbs@math.uio.no}

\date{\today}


\begin{document}

\begin{abstract}
  We analyze monotone finite difference schemes for strongly 
  degenerate convection-diffusion equations in one spatial dimension. These
  nonlinear equations are well-posed within a class of (discontinuous)
  entropy solutions.  We prove that the $L^1$ error between the
  approximate and exact solutions is $\mathcal{O}(\Dx^{1/3})$, where 
  $\Dx$ is the spatial grid parameter. This result should be compared with the classical
  $\mathcal{O}(\Dx^{1/2})$ error estimate for conservation laws
  \cite{Kuznetsov:1976ys}, and a recent estimate of
  $\mathcal{O}(\Dx^{1/11})$ for degenerate convection-diffusion
  equations \cite{Karlsen:2010uq}.
\end{abstract}

\maketitle

\tableofcontents

\allowdisplaybreaks

\section{Introduction}
Nonlinear convection-dominated flow problems arise in a range of
applications, such as fluid dynamics, meteorology, transport of oil
and gas in porous media, electro-magnetism, as well as in many other
applications.  As a consequence it has become a very important
undertaking to construct robust, accurate, and efficient methods for
the numerical approximation of such problems.  Over the years a large
number of stable (convergent) numerical methods have been developed
for linear and nonlinear convection-diffusion equations in which the
``diffusion part" is small, or even vanishing, relative to the
``convection part" of the equation. There is a large literature on
this topic, and we will provide a few relevant references later.

One central but exceedingly difficult issue relating to numerical
methods for convection-diffusion equations, is the derivation of (a
priori) error estimates that are robust in the singular limit as the
diffusion coefficient vanishes, avoiding the exponential growth of
error constants.  This problem has been resolved only partly in
special situations, such as for linear equations or in the completely
degenerate case of no diffusion (scalar conservation laws). For
general nonlinear equations containing both convection and
(degenerate) diffusion terms this is a long standing open problem in
numerical analysis.

This paper makes a small contribution to this general problem by
deriving an error estimate for a class of simple difference schemes
for nonlinear and strongly degenerate convection-diffusion problems of
the form
\begin{equation}\label{MainProblem}
  \begin{cases}
    \partial_tu + \partial_xf(u) = \partial_x^2A(u), & (x,t) \in \Pi_T,\\
    u(0,x) = u^0(x), & x \in \R,
  \end{cases}
\end{equation}
where $\Pi_T = \R \times (0,T)$ for some fixed final time $T>0$, and
$u(x,t)$ is the scalar unknown function that is sought. The initial
function $u_0:\R\to\R$ is a given integrable and bounded function,
while the convection flux $f:\R\to\R$ and the diffusion function
$A:\R\to\R$ are given functions satisfying
\begin{equation*}%\label{eq:flux-ass}
  \text{$f,A$ locally $C^1$; $A(0)=0$; $A$ nondecreasing.}
\end{equation*}

The moniker \emph{strongly degenerate} means that we allow $A'(u) = 0$
for all $u$ in some interval $[\alpha,\beta]\subset \R$. Thus, the
class of equations becomes very general, including purely hyperbolic
equations (scalar conservation laws)
\begin{equation}\label{eq:CL}
  \partial_tu + \partial_xf(u) =0,
\end{equation}
as well as nondegenerate (uniformly parabolic) equations, such as the
heat equation $\partial_tu = \partial_x^2 u$, and point-degenerate
diffusion equations, such as the heat equation with a power-law
nonlinearity: $\partial_tu = \partial_x ( u^m \partial_x u)$, which is
degenerate at $u=0$.

Whenever the problem \eqref{MainProblem} is uniformly parabolic (i.e.,
$A' \ge \eta$ for some $\eta > 0$), it is well known that the problem
admits a unique classical (smooth) solution. On the other hand, in the
strongly degenerate case, \eqref{MainProblem} must be interpreted in
the weak sense to account for possibly discontinuous (shock wave)
solutions.  Regarding weak solutions, it turns out that one needs an
additional condition, the so-called \textit{entropy condition}, to
ensure that \eqref{MainProblem} is well-posed.  More precisely, the
following is known: For $u_0\in L^1(\R)\cap L^\infty(\R)$, there
exists a unique solution $u\in C((0,T);L^1(\R^d))$, $u\in
L^\infty(\Pi_T)$ of \eqref{MainProblem} such that $\partial_x A(u)\in
L^2(\Pi_T)$ and for all convex functions $S:\R\to\R$ with $q_S'=f' S'$
and $r_S'=A'S'$,
\begin{equation}\label{eq:entropyineq}
  \partial_t S(u) + \partial_x q_S(u) - \partial_x^2 r_S(u)\le 0 \quad
  \text{in the weak sense on $[0,T)\times \R$}.
\end{equation}
The satisfaction of these inequalities for all convex $S$ is the
entropy condition, and a weak solution satisfying the entropy
condition is called an entropy solution.  The well-posedness of
entropy solutions is a famous result due to Kru{\v{z}}kov
\cite{Kruzkov:1970kx} for conservation laws \eqref{eq:CL}, and a more
recent work by Carrillo \cite{Carrillo:1999hq} extends this to
degenerate parabolic equations \eqref{MainProblem}.  These results are
available in the multi-dimensional context, and we refer to
\cite{Andreianov:2012kx,Dafermos:2010fk} for an overview of the
relevant literature. For uniqueness of entropy solutions in the $BV$
class, see \cite{Volpert:1967vn,Wu:1989ve}.

One traditional way of constructing entropy solutions is by the
vanishing viscosity method, which starts off from classical solutions
to the nondegenerate equation
\begin{equation*}%\label{eq:VV}
  \partial_t u_\eta + \partial_xf(u_\eta) 
  = \partial_x^2A(u_\eta) + \eta \partial_x^2 u_\eta, 
  \qquad \eta>0,
\end{equation*}
and establishes the strong convergence of $u_\eta$ as $\eta\to 0$ by
deriving $BV$ estimates that are independent of $\eta$, see
Vol{$'$}pert and Hudjaev \cite{Volpert:1969fk}.

Besides proving that $u_\eta$ converges in the $L^1$ norm to the
unique entropy solution $u$ of \eqref{MainProblem}, it is possible to
prove the error estimate
\begin{equation}\label{eq:VV-error}
  \norm{u_\eta(\cdot,t)-u(\cdot,t)}_{L^1}\le 
  C \, \sqrt{\eta}, \qquad \text{(whenever $u_0\in BV$)},
\end{equation}
see \cite{Evje:2002aq} (cf.~also \cite{Eymard:2002eu}).  The error
bound \eqref{eq:VV-error} can also be obtained as a consequence of the
more general continuous dependence estimate derived in
\cite{Cockburn:1999fk}, see also \cite{Chen:2005wf,Karlsen:2003za}.

Herein we are interested in the much more difficult problem of
deriving error estimates for numerical approximations of entropy
solutions to convection-diffusion equations. Convergence results
(without error estimates) have been obtained for finite difference
schemes \cite{Evje:2000vn} (see also
\cite{Evje:2000ix,Karlsen:2001ul}); finite volume schemes
\cite{Eymard:2002nx} (see also \cite{Andreianov:2009kx}); operator
splitting methods \cite{Holden:2010fk}; and BGK approximations
\cite{Aregba-Driollet:2003, Bouchut:2000dp}, to mention just a few references. For a
posteriori estimates for finite volume schemes, see
\cite{Ohlberger:2001oq}.

To be concrete in what follows, let us for simplicity assume $f'\ge 0$
and consider the semi-discrete difference scheme
\begin{equation}\label{eq:FDM}
  \frac{d}{dt}u_j(t) + \frac{f(u_j)-f(u_{j-1})}{\Delta x} 
  = \frac{A(u_{j+1})-2A(u_j)+A(u_{j-1})}{\Delta x^2},
\end{equation}
where $u_j(t)\approx u(t,j\Delta x)$ and $\Delta x>0$ is the spatial
mesh size.  Convergence of this scheme can be proved as in the works
\cite{Evje:2000vn,Evje:2000ix}, where explicit and implicit time
discretizations are treated.  Denote by $u_{\Delta x}(x,t)$ the
piecewise constant interpolation of $\left\{u_j(t)\right\}_{j}$.  The
basic question we address in this paper is the following one: Does
there exist a number $r\in (0,1)$ and a constant $C$, independent of
$\Delta x$, such that
\begin{equation}\label{eq:FDM-error}
  \norm{u_{\Delta x}(\cdot,t)-u(\cdot,t)}_{L^1}\le 
  C \, \Delta x^r,
\end{equation}
where $u$ is the unique entropy solution of \eqref{MainProblem}.  We
refer to the number $r$ as the rate of convergence.

In the purely hyperbolic case \eqref{eq:CL} ($A'\equiv 0$), the answer
to this question is a classical result due to Kuznetsov
\cite{Kuznetsov:1976ys}, who proved that the rate of convergence is
$1/2$ for viscous approximations as well as monotone difference
schemes, and this is optimal for discontinuous solutions.  The work of
Kuznetsov \cite{Kuznetsov:1976ys} turned out to be extremely
influential, and by now a large number of related works have been
devoted to error estimation theory for conservation laws.  We refer to
\cite{Cockburn:2003ys} for an overview of the relevant results and
literature.

Unfortunately, the situation is unclear in the degenerate parabolic
case \eqref{MainProblem}. Let us expose some reasons why adding a
nonlinear diffusion term to \eqref{eq:CL} can make the error analysis
significantly more difficult than in the streamlined Kuznetsov theory.
First of all, it is well known that the purely hyperbolic difference
scheme
\begin{equation}\label{eq:upwind}
  \frac{d}{dt}u_j(t) + \frac{f(u_j)-f(u_{j-1})}{\Delta x}=0
\end{equation}
has as a model equation the second order viscous equation
$$
\partial_t u+ \partial_xf(u) = \frac{\Delta x}{2}\partial_x^2 f(u),
$$
an equation that is compatible with the notion of entropy solution for
\eqref{eq:CL}. Indeed, an error estimate for this viscous equation is
highly suggestive for what to expect for the upwind scheme
\eqref{eq:upwind} (this is of course what Kuznetsov proved).  However,
for convection-diffusion equations such as \eqref{MainProblem} the
situation changes. The model equation for \eqref{eq:FDM}
is no longer second order but rather fourth order:
\begin{equation*}%\label{eq:VV4}
  \partial_t u+ \partial_xf(u) 
  = \partial_x^2 A(u) + 
  \frac{\Delta x}{2}\partial_x^2 f(u) 
  - \frac{\Delta x^2}{12} \partial_x^4 A(u);
\end{equation*}
hence the error estimate \eqref{eq:VV-error} appears no longer so
relevant for numerical schemes. Another added difficulty comes 
from the necessity to work with an explicit form of the 
parabolic dissipation term associated with
\eqref{MainProblem}. Indeed, in the analysis one needs to replace
\eqref{eq:entropyineq} by the following more precise entropy equation
\cite{Carrillo:1999hq}
\begin{equation}\label{eq:entropyeq}
  \begin{split}
    & \partial_t \abs{u-c} + \partial_x
    \bigl(\sign(u-c)(f(u)-f(c)\bigr) - \partial_x^2 \abs{A(u)-A(c)} \\
    & \qquad = -\sign'(A(u)-A(c)) \abs{\partial_x A(u)}^2, \qquad c\in
    \R,
  \end{split}
\end{equation}
which is formally obtained multiplying \eqref{MainProblem} by
$\sgn{A(u)-A(c)}$, assuming for the sake of this discussion that
$A'(\cdot)>0$.  The term on the right-hand side is the parabolic
dissipation term, which is a finite (signed) measure and thus very
singular. To illustrate why the parabolic dissipation term is needed,
let $u(y,s)$ and $v(x,t)$ be two solutions satisfying
\eqref{eq:entropyeq}. In the entropy equation for $u(y,s)$ one takes
$c=v(x,t)$, while in the entropy equation for $v(x,t)$ one takes
$c=u(y,s)$. Adding the two resulting equations yields
\begin{align*}
  &(\partial_t +\partial_s) \abs{u-v} + (\partial_x+\partial_y)
  \bigl(\sign(u-v)(f(u)-f(v)\bigr) \\ & \quad -
  (\partial_x^2+\partial_y^2) \abs{A(u)-A(v)} =
  -\sign'(A(u)-A(v))\bigl( \abs{\partial_y A(u)}^2+\abs{\partial_x
    A(v)}^2\bigr),
\end{align*}
By adding $-2\partial_{xy}^2 \abs{A(u)-A(v)}$ to both sides of this
equation, noting that
$$
-2\partial_{xy}^2 \abs{A(u)-A(v)}= 2 \sign'(A(u)-A(v)) \partial_y
A(u) \partial_x A(v),
$$
we arrive at
\begin{equation}\label{eq:uniq-tmp}
  \begin{split}
    &(\partial_t +\partial_s) \abs{u-v} + (\partial_x+\partial_y)
    \bigl(\sign(u-v)(f(u)-f(v)\bigr) \\ & \qquad \qquad \qquad\qquad -
    (\partial_x^2-2 \partial_{xy}^2+\partial_y^2) \abs{A(u)-A(v)} \\ &
    \qquad\quad = -\sign'(A(u)-A(v))\bigl( \abs{\partial_y A(u)}
    -\abs{\partial_x A(v)}\bigr)^2 \\ &\qquad \quad \le 0,
  \end{split}
\end{equation}
from which the contraction property
$\frac{d}{dt}\norm{u(\cdot,t)-v(\cdot,t)}_{L^1}\le 0$ follows
\cite{Carrillo:1999hq}.  Similarly, to obtain error estimates for
numerical methods, it is necessary to derive a ``discrete" version of
\eqref{eq:uniq-tmp} with $v$ replaced by $u_{\Delta x}$. The main
challenge is to suitably replicate at the discrete level the delicate
balance between the two terms in \eqref{eq:uniq-tmp} involving $A$;
the difficulty stems from the lack of a chain rule for finite
differences.

Despite the mentioned difficulties, we will in this paper prove that
there exists a constant $C$, independent of $\Delta x$, such that for
any $t>0$,
$$
\norm{u_{\Delta x}(\cdot,t)-u(\cdot,t)}_{L^1}\le C \, \Delta
x^{\frac{1}{3}}.
$$
The only other work we are aware of that provides $L^1$ error
estimates for numerical approximations of \eqref{MainProblem} is
\cite{Karlsen:2010uq}; therein \eqref{eq:FDM-error} is established
with $r=\frac{1}{11}$; if $A$ is a linear function, then the
convergence rate is the usual one, namely $r=\frac12$.  In addition to
the semi-discrete scheme \eqref{eq:upwind}, we will prove similar
results for fully discrete (implicit and explicit) difference schemes.

Roughly speaking, the reason is two-fold for why we can significantly
improve the result in \cite{Karlsen:2010uq}.  First, we are herein
able to provide a more faithful analog of \eqref{eq:uniq-tmp} at the
discrete level.  Second, since $\sign'(\cdot)$ is singular, one has to
work with a Lipschitz continuous approximation $\signe(\cdot)$ of the
sign function $\sign(\cdot)$. The use of this approximation breaks the
symmetry of the corresponding entropy fluxes, and introduces new error
terms that depend on the parameter $\eps$; the process of ``balancing"
terms involving $\Delta x$ and $\eps$ lowers the convergence rate (to
$r=\frac{1}{11}$) \cite{Karlsen:2010uq}.  In the present paper we are
able to dispense with this balancing act. Indeed, we show that it is
possible to send $\eps\to 0$ independently of $\Delta x$.

The remaining part of this paper is organized as follows: In Section
\ref{sec:prelim} we list some relevant a priori estimates satisfied by
viscous approximations and entropy solutions, and provide a definition
of entropy solutions.  The semi-discrete difference scheme is defined
and proved to be well-posed in Section \ref{sec: the finite difference
  scheme}.  We also list several relevant a priori estimates.  Section
\ref{sec: the error estimate} is devoted to the proof of the error
estimate.  In Section~\ref{sec: error estimate implicit} we show that
the proof in Section~\ref{sec: the error estimate} can be adapted to a
fully discrete scheme that is implicit in the time variable.  In fact,
we go through all the steps of the proof and provide the details where
there are considerable differences between the two cases. In
Section~\ref{sec: error estimate explicit} the explicit version of the
scheme is treated. We end the paper with a few 
concluding remarks in Section \ref{sec:final}.

\section{Preliminary material}\label{sec:prelim}

Set $A^\eta(u) := A(u) + \eta u$ for any fixed $\eta > 0$, and
consider the uniformly parabolic problem
\begin{equation}\label{NondegEqu}
  \begin{cases}
    u^\eta_t + f(u^\eta)_x = A^\eta(u^\eta)_{xx}, & (x,t) \in \Pi_T,\\
    u^\eta(x,0) = u^0(x), & x \in \R.
  \end{cases}
\end{equation}
It is well known that \eqref{NondegEqu} admits a unique classical
(smooth) solution.

We collect some relevant (standard) a priori estimates in the next
three lemmas.

\begin{lemma}\label{lem:1.1}
  Suppose $u^0\in L^\infty(\R)\cap L^1(\R) \cap BV(\R)$, and let
  $u^\eta$ be the unique classical solution of \eqref{NondegEqu}.
  Then for any $t>0$,
  \begin{align*}
    \norm{u^\eta(\cdot,t)}_{L^1(\R)} & \le \norm{u^0}_{L^1(\R)},\\
    \norm{u^\eta(\cdot,t)}_{L^\infty(\R)} &\le \|u^0\|_{L^\infty(\R)}, \\
    \abs{u^\eta(\cdot,t)}_{BV(\R)} &\le \abs{u^0}_{BV(\R)}.
  \end{align*}
\end{lemma}

For a proof of the previous and next lemmas, see for example
\cite{Volpert:1969fk}.

\begin{lemma}\label{LipschitzContTime}
  Suppose $u^0\in L^\infty(\R)\cap L^1(\R)$ and $f(u^0)-A(u^0)_x\in
  BV(\R)$. Let $u^\eta$ be the unique classical solution of
  \eqref{NondegEqu}.  Then for any $t_1,t_2>0$,
  \begin{equation*}
    \norm{u^\eta(\cdot,t_2)-u^\eta(\cdot,t_1)}_{L^1(\R)} 
    \le \abs{f(u^0)-A(u^0)_x}_{BV(\R)} \abs{t_2-t_1}.
  \end{equation*}
\end{lemma}

Regarding the following lemma, see \cite{Tassa:1996fk,Evje:2000vn}.

\begin{lemma}\label{FluxDiffLemma}
  Suppose $u^0\in L^\infty(\R)\cap L^1(\R)$ and $f(u^0)-A(u^0)_x\in
  L^\infty(\R)\cap BV(\R)$. Let $u^\eta$ be the unique classical
  solution of \eqref{NondegEqu}.  Then for any $t>0$,
  \begin{align*}
    \norm{f(u^\eta(\cdot,t))-A(u^\eta(\cdot,t))_x}_{L^\infty(\R)} &
    \le
    \norm{f(u^0)-A(u^0)_x}_{L^\infty(\R)}, %\label{ContFluxDiffusionMax}
    \\
    \abs{f(u^\eta(\cdot,t))-A(u^\eta(\cdot,t))_x}_{BV(\R)} &\le
    \abs{f(u^0)-A(u^0)_x}_{BV(\R)}. %\label{ContFluxDiffusionBV}
  \end{align*}
\end{lemma}
Note that $\norm{A(u^\eta)_x}_{L^\infty_t(L^\infty_x)}$ and
$\norm{A(u^\eta)_{xx}}_{L^\infty_t(L^1_x)}$ are bounded independently
of $\eta$ provided that $A(u^0)_x$ is in $BV(\R)$.

The results above imply that $\seq{u^\eta}_{\eta > 0}$ is 
relatively compact in $C([0,T];L^1_{\loc}(\R))$.  
If $u=\lim_{\eta\to 0} u^\eta$, then
\begin{equation*}
  \norm{u^\eta-u}_{L^1(\Pi_T)} \le C\eta^{1/2},
\end{equation*}
for some constant $C$ which does not depend on $\eta$, see
\cite{Evje:2002aq}. Moreover, $u$ is an entropy solution according to
the following definition.

\begin{definition}\label{def:entropydef}
  An entropy solution of the Cauchy problem \eqref{MainProblem} is a
  measurable function $u = u(x,t)$ satisfying:
  \begin{itemize}
  \item[(D.1)] $u \in L^\infty(\Pi_T) \cap C((0,T);L^1(\R))$.
  \item[(D.2)] $A(u) \in L^2((0,T);H^1(\R))$.
  \item[(D.3)] For all constants $c \in \R$ and test functions $0\le
    \test \in C_0^\infty(\R \times [0,T))$, the following entropy
    inequality holds:
    \begin{multline*}
      \iint_{\Pi_T} \abs{u-c}\test_t + \sgn{u-c}(f(u)-f(c))\test_x
      +\abs{A(u)-A(c)}\test_{xx}\dxdt \\
      +\int_{\R}\abs{u_0-c}\test(x,0)\, dx \ge 0.
    \end{multline*}
  \end{itemize}
\end{definition}
The uniqueness of entropy solutions follows from the work
\cite{Carrillo:1999hq}.  Actually, in view of the above a priori
estimates, the relevant functional class is $BV(\Pi_T)$, in which case
we can replace (D.2) by the condition $A(u)_x \in
L^\infty(\Pi_T)$. For a uniqueness result in the $BV$ class, see
\cite{Wu:1989ve}.

\section{Difference scheme}\label{sec: the finite difference scheme}
We start by specifying the numerical flux to be used in the difference
scheme.

\begin{definition}{(Numerical flux)}
  We call a function $F \in C^1(\R^2)$ a two-point numerical flux for
  $f$ if $F(u,u) = f(u)$ for $u \in \R$. If
  \begin{equation*}
    \frac{\partial}{\partial u} F(u,v) \ge 0 
    \quad \text{ and } \quad
    \frac{\partial}{\partial v} F(u,v) \le 0 
  \end{equation*}
  holds for all $u,v \in \R$, we call $F$ monotone.
\end{definition}

Let $F_u$ and $F_v$ denote the partial derivatives of $F$ with respect
to the first and second variable, respectively.  We will also assume
that $F$ is Lipschitz continuous.

Let $\Dx > 0$ and set $x_j = j\Dx$ for $j\in \Z$, and define
\begin{equation*}
  D_\pm \sigma_j = \pm \frac{\sigma_{j \pm 1}-\sigma_j}{\Dx},
\end{equation*}
for any sequence $\seq{\sigma_j}$.

We may now define a semi-discrete approximation of the solution to
\eqref{MainProblem} as the solution to the (infinite) system of
ordinary differential equations
\begin{equation}\label{SemiDiscEqu}
  \begin{cases}
    \frac{d}{dt}u_j(t) + \Dm F_{j+1/2} = \Dm\Dp A(u_j),\ &t>0,\\
    u_j(0) = \frac{1}{\Dx}\int_{I_j}u^0(x) \, dx,
  \end{cases}\qquad j \in \Z,
\end{equation}
where $F_{j+1/2} = F(u_j,u_{j+1})$ is a numerical flux function and
$I_j =(x_{j-1/2},x_{j+1/2}]$.

The problem above can be viewed as an ordinary differential equation
in the Banach space $\ell^1(\Z)$ (see, e.g.,~\cite{Ladas:1972fk}).  To
get bounds independent of $\Dx$ we define
\begin{equation*}
  \norm{\sigma}_1 = \Dx \sum_j\abs{\sigma_j} \quad \text{ and } 
  \quad\abs{\sigma}_{BV} = \sum_j\abs{\sigma_{j+1}-\sigma_j} 
  = \norm{\Dp\sigma}_1.
\end{equation*}
If these are bounded we say that $\sigma = \{\sigma_j\}$ is in
$\ell^1$ and of bounded variation. Let $u(t)=\seq{u_j(t)}_{j\in\Z}$,
$u^0 = \seq{u_j(0)}_{j\in\Z}$, and define the operator
$\mathcal{A}:\ell^1 \rightarrow \ell^1$ by $ (\mathcal{A}(u))_j :=
\Dm(F(u_j,u_{j+1})-\Dp A(u_j))$.  Then \eqref{SemiDiscEqu} takes the
following form
\begin{equation*}%\label{SemiDiscEquAbstract}
  \frac{du}{dt} + \mathcal{A}(u) = 0,\ \ t>0,\ \
  u(0) = u^0.
\end{equation*}
This problem has a unique continuously differentiable solution since
$\mathcal{A}$ is Lipschitz continuous for each fixed $\Dx>0$. This
solution defines a strongly continuous semigroup $\mathcal{S}(t)$ on
$\ell^1$.  If $\mathcal{S}$ also satisfies
\begin{equation*}
  \norm{\mathcal{S}(t)u -\mathcal{S}(t)v}_1 \le \norm{u-v}_1 \quad
  \text{ for } \quad u,v \in \ell^1,
\end{equation*}
we say that it is nonexpansive. The next lemma sums up some
important properties of the solutions to \eqref{SemiDiscEqu} (for a
proof see \cite{Evje:1999et}).

\begin{lemma}\label{SemiDiscProp}
  Suppose that $F$ is monotone. Then there exists a unique solution
  $u= \seq{u_j}$ to \eqref{SemiDiscEqu} on $[0,T]$ with the following 
  properties:
  \begin{itemize}
  \item[\textnormal{(a)}] $\norm{u(t)}_1 \le \norm{u^0}_1$.
  \item[\textnormal{(b)}] For every $j \in \mathbb{Z}$ and $t \in
    [0,T]$,
    \begin{equation*}
      \inf_k\seq{u^0_k} \le u_j(t) \le \sup_k\seq{u^0_k}.
    \end{equation*}
  \item[\textnormal{(c)}] $\abs{u(t)}_{BV} \le \abs{u^0}_{BV}$.
  \item[\textnormal{(d)}] If $v = \{v_j\}$ is a another solution with
    initial data $v^0$ then
    \begin{equation*}
      \norm{u(t) -v(t)}_1 \le \norm{u^0-v^0}_1.
    \end{equation*}
  \end{itemize}
\end{lemma}

\begin{lemma}\label{SemiDiscFluxDiffBounds}
  If $F$ is monotone, then
  \begin{align}
    \norm{F(u_j,u_{j+1})-\Dp A(u_j)}_{\ell^\infty} &\le
    \norm{F(u_j^0,u_{j+1}^0)-\Dp A(u_j^0)}_{\ell^\infty},
    \label{FluxDiffBoundInf}\\
    \abs{F(u_j,u_{j+1})-\Dp A(u_j)}_{BV} & \le
    \abs{F(u_j^0,u_{j+1}^0)-\Dp A(u_j^0)}_{BV}.
    \label{FluxDiffBoundBV}
  \end{align}
  Furthermore, $t\mapsto \seq{u_j(t)}_{j\in\Z}$ is $\ell^1$ Lipschitz
  continuous.
\end{lemma}

\begin{proof}
  The proof follows \cite{Evje:2000vn}.  Let $v_j = \Dx\sum_{k \le
    j}\tfrac{d u_k}{dt}$.  Then $v_j$ satisfies
  \begin{equation}\label{defV}
    v_j = \Dx\sum_{k=-\infty}^j
    \Dm(\Dp A(u_k)-F(u_k,u_{k+1})) = \Dp A(u_j)-F(u_j,u_{j+1}),
  \end{equation}
  and we may define $v_j$ for all $t \in [0,T]$. Note that
  $\{v_j(t)\}$ is in $\ell^1$ for all $t$ by
  Lemma~\ref{SemiDiscProp}. Differentiating \eqref{defV} with respect
  to $t$ we obtain
  \begin{align*}
    \frac{dv_j}{dt} &=
    \frac{1}{\Dx}\Bigl[a(u_{j+1})\frac{du_{j+1}}{dt}-a(u_j)\frac{du_j}{dt}\Bigr]
    \\ &\qquad - F_u(u_j,u_{j+1})\frac{du_j}{dt} -
    F_v(u_j,u_{j+1})\frac{du_{j+1}}{dt},
  \end{align*}
  where $a(u)=A'(u)$.  Note that $\Dm v_j = \tfrac{d u_j}{dt}$ and
  $\Dp v_j =\tfrac{du_{j+1}}{dt}$. Therefore
  \begin{multline}\label{VDerExpression}
    \frac{d v_j}{dt}
    = \left(\frac{1}{\Dx}a(u_{j+1})-F_v(u_j,u_{j+1})\right)\Dp v_j \\
    - \left(\frac{1}{\Dx}a(u_j) + F_u(u_j,u_{j+1})\right)\Dm v_j.
  \end{multline}
  Assume $v_{j_0}(t_0)$ is a local maximum in $j$. Then $\Dp
  v_{j_0}(t_0) \le 0$ and $\Dm v_{j_0}(t_0) \ge 0$ so
  $\tfrac{v_{j_0}}{dt}(t_0) \le 0$ since $F$ is monotone. Similarly,
  if $v_{j_0}(t_0)$ is a local minimum in $j$, then
  $\tfrac{v_{j_0}}{dt}(t_0)\ge 0$. Then inequality
  \eqref{FluxDiffBoundInf} follows by the fact that $\seq{v_j(t)} \in
  \ell^1$.  Consider \eqref{FluxDiffBoundBV}. We want to show that
  $\tfrac{d}{dt}\left(|v(t)|_{BV}\right) \le 0$.  Now,
  \begin{equation*}
    \frac{d}{dt}\Bigl(\sum_j\abs{v_{j+1}-v_{j}}\Bigr) 
    =\sum_j \sgn{v_{j+1}-v_j}\frac{d}{dt}\left(v_{j+1}-v_j\right), 
  \end{equation*}
  so we may use \eqref{VDerExpression}. Thus
  \begin{align*}
    & \frac{d}{dt}\abs{v(t)}_{BV} \\ & \quad =
    \sum_j\left(\frac{1}{\Dx}a(u_{j+2})-F_v(u_{j+1},u_{j+2})\right)
    \left(\Dp v_{j+1}\right)\sign(v_{j+1}-v_j)\\
    &\quad \quad - \sum_j\left(\frac{1}{\Dx}a(u_{j+1})
      + F_u(u_{j+1},u_{j+2})\right)\abs{\Dp v_j} \\
    &\quad \quad -
    \sum_j\left(\frac{1}{\Dx}a(u_{j+1})-F_v(u_j,u_{j+1})\right)
    \abs{\Dp v_j} \\
    &\quad \quad + \sum_j\left(\frac{1}{\Dx}a(u_j)
      + F_u(u_j,u_{j+1})\right)((\Dm v_j)\sign(v_{j+1}-v_j)) \\
    &\quad =\sum_j
    \left(\frac{1}{\Dx}a(u_{j+1})-F_v(u_j,u_{j+1})\right)
    \left[\left(\Dp v_{j}\right)\sgn{v_{j}-v_{j-1}}-\abs{\Dp v_{j}}\right]\\
    &\quad\quad +\sum_j\left(\frac{1}{\Dx}a(u_j) +
      F_u(u_j,u_{j+1})\right)
    \left[(\Dm v_j)\sign(v_{j+1}-v_j)-\abs{\Dm v_j}\right]\\
    & \quad \le 0,
  \end{align*}
  since $a(u)>0$, $F_v\le 0$, and $F_u\ge 0$. Given the preceding
  estimates, the $\ell^1$ Lipschitz 
  continuity is straightforward to prove.
\end{proof}

It turns out that we need more conditions on $F$ than mere
monotonicity.
\begin{definition}\label{def:entropyflux}
  Given an entropy pair $(\psi,q)$ and a numerical flux $F$, we define
  $Q \in C^1(\R^2)$ by
  \begin{align*}
    Q(u,u) &= q(u), \\
    \frac{\partial}{\partial v}Q(v,w) &
    = \psi'(v)\frac{\partial}{\partial v}F(v,w), \\
    \frac{\partial}{\partial w}Q(v,w) & =
    \psi'(w)\frac{\partial}{\partial w}F(v,w).
  \end{align*}
  We call $Q$ a numerical entropy flux.
\end{definition}

The next lemma gives a sufficient condition on the numerical flux to
ensure that there exists a numerical entropy flux.

\begin{lemma}\label{existenceOfPotential}
  Given a two-point numerical flux $F$, assume that there exist $C^1$
  functions $F_1, F_2$ such that
  \begin{equation}\label{splitFlux}
    F(u,v) = F_1(u) + F_2(v), \qquad 
    F_1'(u) + F_2'(u) = f'(u), 
  \end{equation}
  for all relevant $u$ and $v$. Then there exists a numerical entropy
  flux $Q$ for any entropy flux pair $(\psi,q)$.
\end{lemma}

\begin{proof}
  Let $(\psi,q)$ be an entropy pair. Then $q$ has the form
  \begin{equation*}
    q(u) = \int_c^u \psi'(z)f'(z) \dz + C,
  \end{equation*}
  for some constant $C$. Define $Q$ by
  \begin{equation}\label{defNumericalFluxPotential}
    Q(u,v) = \int_c^u\psi'(z)F_1'(z) \dz 
    + \int_c^v \psi'(z)F_2'(z) \dz + C.
  \end{equation}
  It is easily verified that $Q$ is a numerical entropy flux.
\end{proof}
 
Let us list a few numerical flux functions to which Lemma
\ref{existenceOfPotential} applies.

\begin{example}[Engquist-Osher flux]
  Let
  \begin{equation*}
    f_+'(s) = \max (f'(s),0) \qquad \text{and} \qquad 
    f_-'(s) = \min (f'(s),0). 
  \end{equation*}
  Then, in the terminology of Lemma~\ref{existenceOfPotential}, let
  $F(u,v) = F_1(u) + F_2(v)$ with
  \begin{equation*}
    F_1(u) = f(0) + \int_0^u f_+'(s) \, ds 
    \qquad \text{and} \qquad 
    F_2(v) = \int_0^v f_-'(s) \, ds.
  \end{equation*}
  It is easily seen that the criteria given in
  Lemma~\ref{existenceOfPotential} are satisfied, and $F$ is also
  clearly monotone.
\end{example}

\begin{example} 
  Let $a,b \in \R$ and define
  \begin{equation*}
    F_1(u) = af(u) + bu \qquad 
    \text{and} \qquad F_2(v) = (1-a)f(v)-bv.
  \end{equation*}
  Note that $F(u,v) = F_1(u) + F_2(v)$ is monotone if
  \begin{equation*}
    a \inf_u\{f'(u)\} \ge -b \quad  
    \text{and} \quad (1-a)\sup_u\{f'(u)\} \le b.
  \end{equation*}
  This example includes both the upwind scheme and the Lax-Friedrichs
  scheme.
\end{example}

From a more general point of view we may consider any 
flux splitting, that is, $f(u) = f^+(u) + f^-(u)$ 
with $(f^+)'(u) \ge 0$ and $(f^-)'(u) \le 0$ for all $u \in \R$.  
Then the numerical flux
$$
F(u,v) = f^+(u) + f^-(v)
$$
satisfies the assumptions of Lemma~\ref{existenceOfPotential}. Note
also that any convex combination of numerical flux functions which
satisfy the hypothesis of Lemma~\ref{existenceOfPotential}, itself
satisfies the assumptions of the lemma.
  
If \eqref{splitFlux} holds, then we have a representation of $Q$ given
by \eqref{defNumericalFluxPotential}. It follows that
\begin{equation}\label{eq:NumFluxPres}
  Q(u,v) = q(u) + \int_u^v \psi'(z)F_2'(z) \dz.
\end{equation}
Note that we may obtain another representation depending on $F_1$ by
splitting up the first integral.

\section{Error estimate}\label{sec: the error estimate}
Let $\seq{u_j}_{j\in\Z}$ be the solution to \eqref{SemiDiscEqu}. We
associate with it the piecewise constant function
\begin{equation}\label{eq:udxdef}
  u_\Dx(x,t)=u_j(t) \quad  \text{for $x\in I_j$.}
\end{equation}
To derive the error estimate we need many of the uniform bounds from
Sections~\ref{sec:prelim} and \ref{sec: the finite difference scheme}.
For these estimates to hold independently of $\Dx$, we make the
following assumptions on the initial data $u^0$:
\begin{itemize}
\item[(i)] $u^0 \in L^1(\R)\cap L^\infty(\R)\cap BV(\R)$.
\item[(ii)] $A(u^0)_x\in BV(\R)$.
\end{itemize}

We may now state the theorem.
\begin{theorem}\label{MainResultSemiDisc}
  Let $u$ be the entropy solution to \eqref{MainProblem} and
  $\seq{u_j(t)}_{j\in\Z}$ solve the semi-discrete difference scheme
  \eqref{SemiDiscEqu}.  If $u^0$ satisfies \textnormal{(i)} and
  \textnormal{(ii)} above, then for all sufficiently small $\Dx$,
  \begin{equation*}%\label{eq:MainResultSemiDisc}
    \norm{u_\Dx(\cdot,t)-u(\cdot,t)}_{L^1(\R)}
    \le \norm{u_\Dx^0-u^0}_{L^1(\R)} + C_T \Dx^{\frac{1}{3}}, 
    \qquad t \in [0,T],
  \end{equation*}
  where the constant $C_T$ depends on $A$, $f$, $u^0$, and $T$, but
  not on $\Dx$.
\end{theorem}

Let us define some of the functions we are going to work with.  First,
we will use the following approximation of the $\sign$ function:
\begin{equation*}
  \signe (\sigma) 
  = 
  \begin{cases}
    \sin(\frac{\pi\sigma}{2\eps}) &\text{for $\abs{\sigma}<\eps$},\\
    \sgn{\sigma} &\text{otherwise,}
  \end{cases}
\end{equation*}
where $\eps > 0$. Note that $\signe$ is continuously differentiable
and non-decreasing. We define
 $$
 \abseps{u} = \int_0^u \signe(z) \dz.
$$
Furthermore, we introduce an entropy pair $(\psie,\qe)$ defined by
\begin{align*}
  \psie(u,c) &= \int_c^u \signe(A(z)-A(c)) \dz, \\
  \qe(u,c) &= \int_c^u \psie'(z,c)f'(z)\dz = \int_c^u
  \signe(A(z)-A(c))f'(z) \dz,
\end{align*}
where $\psie'$ is the derivative with respect to the first variable.

\begin{lemma}\label{lemma:EntropyCalcCont}
  Suppose $A' > 0$. Let $u = u(y,s)$ be the classical solution of
  \eqref{MainProblem}. Then for any constant $c \in \R$,
$$
\partial_s\psie(u,c) + \partial_y \qe(u,c)
- \partial_y^2\abseps{A(u)-A(c)} = - \partial_y\psie'(u,c)\partial_y
A(u).
$$
\end{lemma}

\begin{proof}
  Multiply equation \eqref{MainProblem} by $\psie'(u,c)$ to obtain
  \begin{displaymath}
    \partial_s\psie(u,c) + \partial_y \qe(u,c) = \psie'(u,c)\partial_y^2 A(u).
  \end{displaymath}
  The term on the right may be rewritten according to
  \begin{displaymath}
    \partial_y(\psie'(u,c)\partial_y A(u)) 
    = \partial_y\psie'(u,c)\partial_y A(u) 
    + \psie'(u,c)\partial_y^2 A(u).
  \end{displaymath}
  By the chain rule
$$
\partial_y(\psie'(u,c)\partial_y A(u))
= \partial_y^2\abseps{A(u)-A(c)}.
$$
Combining these equalities proves the lemma.
\end{proof}

The next lemma is a simple identity 
taken from \cite{Cockburn:1996}.
\begin{lemma}\label{lemma:ChainRule2}
  For any differentiable function $g$ and all real numbers $a,b,c$,
  \begin{align*}%\label{eq:ChainForm}
    \psie'(a,c)(g(b)-g(a)) & = \int_c^b \psie'(z,c)g'(z)\,dz-\int_c^a
    \psie'(z,c)g'(z) \,dz \\ & \qquad + \int_a^b
    \psie''(z,c)(g(z)-g(b))\dz.
  \end{align*}
\end{lemma}

\begin{proof}
  Integration by parts yields
$$
\psie'(\zeta,c)(g(\zeta)-g(b)) = \int_c^\zeta \psie'(z,c)g'(z) \,dz +
\int_c^\zeta \psie''(z,c)(g(z)-g(b))\dz,
$$
for any $\zeta \in \R$. Take the two equations obtained by taking
$\zeta = a$ and $\zeta = b$ and subtract one from the other.
\end{proof}

\begin{lemma}\label{lemma: EntropyCalcDiscrete}
  Let $u_j$ be the solution to \eqref{SemiDiscEqu}.  Then for all $c
  \in \R$
  \begin{align*}%\label{SemiDiscEqu1}
    & \partial_t\psie(u_j,c) + \Dm Q^c(u_j,u_{j+1}) -\Dm\Dp
    \abseps{A(u_j)-A(c)} \\ & \qquad \le
    -\frac{1}{(\Dx)^2}\int^{u_j}_{u_{j+1}}
    \psie''(z,c)(A(z)-A(u_{j+1}))\dz \\ & \qquad \qquad
    -\frac{1}{(\Dx)^2}\int^{u_j}_{u_{j-1}}
    \psie''(z,c)(A(z)-A(u_{j-1}))\dz,
  \end{align*}
  where $Q^c(u,v) := Q^c_1(u) + Q^c_2(v)$,
  \begin{equation*}
    Q^c_1(u) := \int_c^u\psie'(z,c)F_1'(z) \dz, \quad 
    Q^c_2(v) := \int_c^v \psie'(z,c)F_2'(z) \dz,
  \end{equation*}
  for all real numbers $u$ and $v$.
\end{lemma}

\begin{proof}
  From \eqref{SemiDiscEqu} it follows that
$$
\psie'(u_j,c)\partial_t u_j + \psie'(u_j,c)\Dm F(u_j,u_{j+1}) =
\psie'(u_j,c)\Dm\Dp A(u_j).
$$
Note that
$$
\psie'(u_j,c)\Dm F(u_j,u_{j+1}) = \psie'(u_j,c)\Dm F_1(u_j) +
\psie'(u_j,c)\Dp F_2(u_j),
$$
and so we may apply Lemma~\ref{lemma:ChainRule2}.  Let $g = F_1$. Then
we obtain
$$
\psie'(u_j,c)\Dm F_1(u_j) = \Dm Q_1^c(u_j) - \frac{1}{\Dx}
\int_{u_j}^{u_{j-1}} \psie''(z,c)(F_1(z)-F_1(u_{j-1}))\dz.
$$
Similarily, let $g = F_2$ to obtain
$$
\psie'(u_j,c) \Dp F_2(u_j) = \Dp Q_2^c(u_j) + \frac{1}{\Dx}
\int_{u_j}^{u_{j+1}}\psie''(z,c)(F_2(z)-F_2(u_{j+1}))\,dz.
$$
Finally, apply lemma \ref{lemma:ChainRule2} twice with $g = A$.
Adding the equations we obtain
\begin{align*}
  \psie'(u_j,c) & \left(A(u_{j-1}) - 2A(u_j) + A(u_{j+1})\right) \\
  &=\int_{u_j}^{u_{j+1}} \psie'(z,c)A'(z)\dz + \int_{u_j}^{u_{j-1}} \psie'(z,c)A'(z)\dz \\
  &\qquad +\int_{u_j}^{u_{j+1}} \psie''(z,c)(A(z)-A(u_{j+1}))\dz \\
  &\qquad +\int_{u_j}^{u_{j-1}} \psie''(z,c)(A(z)-A(u_{j-1}))\dz.
\end{align*}
Note that
\begin{align*}
  &\left[\int_{u_j}^{u_{j+1}} \psie'(z,c)A'(z)\dz +
    \int_{u_j}^{u_{j-1}} \psie'(z,c)A'(z)\dz\right] \\ &
  =\left[\int_{u_j}^{u_{j+1}} \frac{\partial}{\partial
      z}\abseps{A(z)-A(c)}\dz +
    \int_{u_j}^{u_{j-1}}\frac{\partial}{\partial
      z}\abseps{A(z)-A(c)}\dz \right] \\ &
  =\left[\abseps{A(u_{j-1})-A(c)}-2\abseps{A(u_j)-A(c)} +
    \abseps{A(u_{j+1})-A(c)}\right].
\end{align*}
Combining the above computations we obtain
$$
\partial_t\psie(u_j,c) + \Dm Q^c(u_j,u_{j+1}) -\Dm\Dp
\abseps{A(u_j)-A(u)} = -E^c(u_{j-1},u_j,u_{j+1}),
$$
where
\begin{align*}
  E^c(u_{j-1},u_j,u_{j+1})
  & = \frac{1}{\Dx} \int^{u_j}_{u_{j-1}} \psie''(z,c)(F_1(z)-F_1(u_{j-1}))\dz \\
  & \qquad  -\frac{1}{\Dx} \int^{u_j}_{u_{j+1}}\psie''(z,c)(F_2(z)-F_2(u_{j+1}))\dz \\
  & \qquad +\frac{1}{(\Dx)^2}\int^{u_j}_{u_{j+1}} \psie''(z,c)(A(z)-A(u_{j+1}))\dz \\
  & \qquad +\frac{1}{(\Dx)^2}\int^{u_j}_{u_{j-1}}
  \psie''(z,c)(A(z)-A(u_{j-1}))\dz.
\end{align*}
The result follows from the monotonicity of $F$.
\end{proof}

We shall need the next lemma, which deals with a mixed term, in order to
carry out the ``second order" doubling-of-the-variables argument.

\begin{lemma}\label{lemma:MixedTerm}
  Let $\seq{u_j}$ be some sequence and $u$ some differentiable
  function of $y$.  Then
  \begin{align*}
    & \left(\frac{1}{\Dx}\int_{u_{j-1}}^{u_j}\psie''(z,u)\dz +
      \frac{1}{\Dx}\int_{u_j}^{u_{j+1}}\psie''(z,u)\dz
    \right)\partial_yA(u) \\ & \qquad \qquad = -\partial_y(\Dm +
    \Dp)\abseps{A(u_{j})-A(u)}.
  \end{align*}
\end{lemma}

\begin{proof}
  Let $a,b$ be fixed real numbers. Then
  \begin{equation*}
    \begin{split}
      & \int_{a}^{b}\psie''(z,u)\dz A(u)_y \\
      & \quad = \int_{a}^{b}\signe'(A(z)-A(u))A(u)_yA'(z)\dz \\
      & \quad = -\frac{\partial}{\partial y}\left(\int_{a}^{b}\signe(A(z)-A(u))A'(z)\dz\right) \\
      & \quad = -\frac{\partial}{\partial
        y}\left(\abseps{A(b)-A(u)}-\abseps{A(a)-A(u)}\right).
    \end{split}
  \end{equation*}
  Let $a=u_{j-1}$, $b=u_j$ and $a=u_j$, $b=u_{j+1}$. Then add up the
  resulting equations and divide by $\Dx$.
\end{proof}

We are now in a position to carry out the doubling-of-the-variables
argument.

\begin{lemma}\label{lemma:DoubelingLemma}
  Suppose $A' > 0$. Let $u = u(y,s)$ be the classical solution to
  \eqref{MainProblem} and let $\seq{u_j} = \seq{u_j(t)}$ be the
  solution to \eqref{SemiDiscEqu}. Then
  \begin{align*}
    & \partial_t\psie(u_j,u) + \partial_s\psie(u,u_j) + \partial_y
    \qe(u,u_j) + \Dm Q^u(u_j,u_{j+1}) \\ & \qquad -(\partial_y^2
    +\partial_y(\Dm +\Dp) + \Dm\Dp)\abseps{A(u_j)-A(u)}) \le -\Ee_j,
  \end{align*}
  where $\Ee_j := \Ee[u](u_{j-1},u_j,u_{j+1})$ with
  \begin{align*}
    \Ee[u](u_{j-1},u_j,u_{j+1})
    &:= \frac{1}{(\Dx)^2}\int^{u_j}_{u_{j+1}} \psie''(z,u)(A(z)-A(u_{j+1}))\dz \\
    &\qquad  + \frac{1}{(\Dx)^2}\int^{u_j}_{u_{j-1}} \psie''(z,u)(A(z)-A(u_{j-1}))\dz \\
    & \qquad - \frac{1}{\Dx}\int_{u_{j-1}}^{u_j}\psie''(z,u)\dz\, \partial_yA(u)  \\
    & \qquad - \frac{1}{\Dx}\int_{u_j}^{u_{j+1}}\psie''(z,u)\dz\, \partial_yA(u) \\
    & \qquad + \partial_y\psie'(u,u_j)\partial_y A(u).
  \end{align*}
\end{lemma}

\begin{proof}
  Let $c = u_j$ in Lemma~\ref{lemma:EntropyCalcCont} and $c = u$ in
  Lemma~\ref{lemma: EntropyCalcDiscrete}.  Then add up the equations
  together with Lemma~\ref{lemma:MixedTerm}.
\end{proof}

\begin{remark}
Note that $\Ee_j$ is a function of $y,s,t$.
\end{remark}

In what follows it will be necessary to work with the piecewice
constant approximation defined in \eqref{eq:udxdef}. To do this we
introduce some new notation.  Let the shift operator $S_\sigma$ be
defined for any $\test:\Pi_T\to \R$ by
$$
S_\sigma\test(x,t) = \test(x + \sigma,t),
$$
and the difference quotient be defined by
$$
D_{\pm} \test = \pm \frac{S_{\pm \Dx}\test-\test}{\Dx}.
$$
Note that for any two functions $u, v$ of $x$ we have $\Dp(uv) = S_\Dx
u \Dp v + (\Dp u)v$.  If $uv$ has compact support it follows that
$$
\int_\R(\Dp u)v \,dx = -\int_\R u\Dm v\, dx.
$$
We will refer to these identities as the Leibniz rule for difference
quotients and integration by parts for difference quotients. We will
frequently integrate over the domain $\Pi_T^2$. To avoid writing four
integral signs we will in general write one for each domain $\Pi_T$
and let $dX = dxdtdyds$.

\begin{lemma}\label{lemma:MainLemmaTestFunc}
  Suppose $A' > 0$. Let $u_\Dx=u_\Dx(x,t)$ be defined by
  \eqref{eq:udxdef}, and let $u=u(y,s)$ be the classical solution of
  \eqref{MainProblem}.  Let $\rho \in C^\infty_0(\R)$ satisfy
  \begin{equation*}
    \mathrm{supp}(\rho) \subset [-1,1], \quad
    \rho(-\sigma)=\rho(\sigma),\quad \rho(\sigma) \ge 0, \quad 
    \int_\R \rho(\sigma) \, d\sigma = 1, 
  \end{equation*}
  and set
  \begin{equation*}
    \omega_r(x)=\frac1r \rho\left(\frac{x}{r}\right),\quad 
    \rho_\alpha(\xi)=\frac{1}{\alpha}\rho\left(\frac{\xi}{\alpha}\right), 
    \quad  \rho_{r_0}(t)=\frac{1}{r_0}\rho\left(\frac{t}{r_0}\right),
  \end{equation*}
  for positive (small) $r$, $\alpha$ and $r_0$.  Let $\nu$ and $\tau$
  be such that $0<\nu<\tau<T$ and define
  \begin{equation*}
    \psi^{\alpha}(t) := H_{\alpha}(t-\nu)-H_{\alpha}(t-\tau), \quad
    H_{\alpha}(t) = \int_{-\infty}^t \rho_{\alpha}(\xi)\, d\xi.  
  \end{equation*}
  Let
  \begin{equation*}%\label{eq:testdef}
    \test(x,t,y,s) = \psi^\alpha(t)\omega_r(x-y)\rho_{r_0}(t-s).
  \end{equation*}
  To ensure $\test_{|t = 0} \equiv 0$, $\test_{|s=0} \equiv 0$, we
  choose $\nu$ and $\tau$ such that $0 < r_0 < \min(\nu,T-\tau)$ and
  $0<\alpha<\min(\nu-r_0,T-\tau-r_0)$. Then
  \begin{equation}\label{MainIneq2}
    \begin{split}
      \intPi &\abs{u_\Dx-u}\rho_{\alpha}(t-\nu)\omega_r\rho_{r_0} \dX  \\
      &+\intPi \sgn{u_\Dx-u}\left(f(u_\Dx)-f(u)\right) \left(\Dp\test +\test_{y}\right)\dX \\
      &+\intPi \left(\int_{u_\Dx}^{S_\Dx u_{\Dx}} \sgn{z-u}F_2'(z)\,dz\right) \,\Dp\test \dX \\
      &+\intPi \abs{A(u_\Dx)-A(u)} \left(\Dm\Dp\test + (\Dp + \Dm)\test_{y} + \test_{yy}\right) \dX  \\
      &\qquad \ge \intPi
      \abs{u_\Dx-u}\rho_{\alpha}(t-\tau)\omega_r\rho_{r_0} \dX +
      \liminf_{\varepsilon \downto 0} \intPi \Ee_\Dx \test \dX,
    \end{split}
  \end{equation}
  where $\Ee_\Dx(x,t,y,s) = \Ee_j(t,y,s)$ for $x \in I_j$.
\end{lemma}

\begin{remark}
  Note that both
$$
\test_{x} + \test_{y} = 0 \quad \mbox{and} \quad \test_{xx} +
2\test_{xy} + \test_{yy} = 0.
$$
In equation \eqref{MainIneq2} these expressions appear with difference
quotients instead of $x$-derivatives. We expect that these equalities
turns into good approximations as long as $\Dx$ tends relatively fast
to zero compared to $r$.  We will show that this is the case in what
follows.
\end{remark}

\begin{proof}
  By Lemma~\ref{lemma:DoubelingLemma} it follows that
  \begin{align*}
    & \partial_t\psie(u_\Dx,u) + \partial_s\psie(u,u_\Dx) + \partial_y
    \qe(u,u_\Dx) + \Dm Q^u(u_\Dx,S_\Dx u_\Dx) \\ & \qquad
    -( \partial_y^2 +\partial_y(\Dm +\Dp) +
    \Dm\Dp)\abseps{A(u_\Dx)-A(u)} \le -\Ee_\Dx,
  \end{align*}
  for all $(x,t,y,s) \in \Pi_T^2$. Let us multiply with $\test$ and
  integrate over $\Pi_T^2$. Using both ordinary integration by parts
  and integration by parts for difference quotients, we obtain
  \begin{align*}
    & \intPi \psie(u_\Dx,u)\test_t + \psie(u,u_\Dx)\test_s \dX \\
    & \qquad
    +\intPi  \qe(u,u_\Dx)\test_y +  Q^u(u_\Dx,S_\Dx u_\Dx)\Dp\test \dX \\
    & \qquad
    +\intPi \abseps{A(u_\Dx)-A(u)}(\test_{yy}  +(\Dm +\Dp)\test_y + \Dm\Dp\test) \dX \\
    & \qquad \qquad \ge \intPi \Ee_\Dx \test \dX.
  \end{align*}
We want to take the limit as $\varepsilon \downto 0$.  Consider the 
first term on the left. By the dominated 
convergence theorem, for any $a,b \in \R$,
$$
\lim_{\varepsilon \downto 0} \psie(a,b) = \lim_{\varepsilon \downto 0}
\int_{b}^a \signe(A(z)-A(b))\dz = \abs{a - b},
$$
since $A' > 0$. It follows that
\begin{displaymath}
  \lim_{\varepsilon \downto 0} \psie(u_\Dx,u) = \lim_{\varepsilon \downto 0}\psie(u,u_\Dx) = |u_\Dx-u|.
\end{displaymath}
Furthermore,
$$
(\test_t + \test_s)(x,t,y,s) =
(\rho_\alpha(t-\nu)-\rho_{\alpha}(t-\tau))\omega_r(x-y)\rho_{r_0}(t-s),
$$
so by the dominated convergence theorem
\begin{multline*}
  \lim_{\varepsilon \downto 0} \intPi \psie(u_\Dx,u)\test_t +
  \psie(u,u_\Dx)\test_s\dX
  \\
  = \intPi \abs{u_\Dx-u}\rho_{\alpha}(t-\nu)\omega_r\rho_{r_0} \dX -
  \intPi \abs{u_\Dx-u}\rho_{\alpha}(t-\tau) \omega_r\rho_{r_0}\dX.
\end{multline*}
Consider the second term on the left. By \eqref{eq:NumFluxPres} we obtain
$$
Q^u(u_\Dx,S_\Dx u_\Dx) = \qe(u_\Dx,u) + \int_{u_\Dx}^{S_\Dx u_\Dx}
\signe(A(z)-A(u))F_2'(z)\dz.
$$
Since $A' > 0$,
\begin{align*}
  \lim_{\varepsilon \downto 0} \qe(u_\Dx,u)
  &= \lim_{\varepsilon \downto 0} \int^{u_\Dx}_u \signe(A(z)-A(u_\Dx))f'(z)\dz \\
  &= \sgn{u_\Dx-u}(f(u_\Dx)-f(u)).
\end{align*}
It follows that
\begin{align*}
  \lim_{\varepsilon \downto 0}Q^u(u_\Dx,S_\Dx u_\Dx) & =
  \sgn{u_\Dx-u}(f(u_\Dx)-f(u)) \\ & \qquad + \int_{u_\Dx}^{S_\Dx
    u_\Dx} \sgn{z-u}F_2'(z)\dz.
\end{align*}
As above
\begin{align*}
  \lim_{\varepsilon \downto 0} \qe(u,u_\Dx) &= \sgn{u-u_\Dx}(f(u)-f(u_\Dx)) \\
  &= \sgn{u_\Dx-u}(f(u_\Dx)-f(u)).
\end{align*}
Hence, again by the dominated convergence theorem,
\begin{align*}
  & \lim_{\varepsilon \downto 0}\intPi \qe(u,u_\Dx)\test_y
  +  Q^u(u_\Dx,S_\Dx u_\Dx)\Dp\test \dX \\
  & \qquad
  = \intPi  \sgn{u_\Dx-u}(f(u_\Dx)-f(u))(\test_y + \Dp\test)\dX  \\
  & \qquad\qquad + \intPi \left(\int_{u_\Dx}^{S_\Dx u_{\Dx}}
    \sgn{z-u}F_2'(z)\,dz \right) \Dp \test \dX.
\end{align*}
\end{proof}
\begin{lemma}\label{lemma:FatousApplicationToDissip}
  Let $\Ee_\Dx$ and $\test$ be defined in
  Lemma~\ref{lemma:MainLemmaTestFunc}. Then
$$
\liminf_{\varepsilon \downto 0} \intPi \Ee_\Dx \test \dX \ge
\int_{\Pi_T}\liminf_{\varepsilon \downto 0} \left(\int_{\Pi_T} \Ee_\Dx
  \test \dyds \right)\dxdt.
$$
\end{lemma}

\begin{proof}
Let
$$
f_\varepsilon(x,t) := \int_{\Pi_T}\Ee_\Dx \test\dyds
$$
and
$$
h_\varepsilon(x,t) := \int_{\Pi_T} \partial_y(\Dm +
\Dp)\abseps{A(u_\Dx)-A(u)}\test \dyds.
$$
Recall that $\Ee_\Dx(x,t,y,s) = \Ee_j(t,y,s)$ for $x \in I_j$, where
$\Ee_j$ is defined in Lemma~\ref{lemma:DoubelingLemma}. Note that
\begin{align*}
  \Ee_j \geq& -
  \frac{1}{\Dx}\int_{u_{j-1}}^{u_j}\psie''(z,u)\dz\, \partial_yA(u)
  \\ 
  & -
  \frac{1}{\Dx}\int_{u_j}^{u_{j+1}}\psie''(z,u)\dz\, \partial_yA(u),
\end{align*}
so by Lemma~\ref{lemma:MixedTerm} it follows 
that $f_\varepsilon \ge h_\varepsilon$. 
Using integration by parts and the triangle inequality 
we obtain the bound
\begin{equation*}
  \abs{h_\varepsilon} \le \left(\abs{\Dm A(u_\Dx)} + \abs{\Dp
      A(u_\Dx)}\right) \left(\int_{\Pi_T} \abs{\test_y}\dyds\right) =:h.
\end{equation*}
It follows by Lemma~\ref{SemiDiscProp} that $h$ is an integrable
nonnegative function such that $-h \le f_\varepsilon$.  By Fatou's
lemma we obtain
\begin{equation*}
  \liminf_{\varepsilon \downto 0}\int_{\Pi_T}f_\varepsilon \dxdt \ge
  \int_{\Pi_T}\liminf_{\varepsilon \downto 0}f_\varepsilon \dxdt.
\end{equation*}
\end{proof}

Note that as $\varepsilon \downto 0$ the terms in $\Ee_j$ concentrate
on the domains specified by $u \in \interval(u_j,u_{j+1})$, $u \in
\interval(u_{j-1},u_j)$, or $u = u_j$. In order to analyze this limit
we will need the following elementary lemma:

\begin{lemma}\label{DifQuotChainRule}
  Let $\seq{u_j}_{j\in \Z}$ be some sequence in $\mathbb{R}$ and let
  $A:\R \rightarrow \R$ a strictly increasing continuously
  differentiable function. For any $u\in\R$ there exist sequences
  $\seq{\tau^\pm_j}_{j \in \Z},\seq{\theta^\pm_j}_{j \in \Z}$ such
  that for each $j \in \Z$ both $\tau^\pm_j$ and $\theta^\pm_j$ are in
  $\interval(u_j,u_{j \pm 1})$ and
  \begin{align*}
    D_\pm \signe(A(u_j)-A(u)) &= \signe'(A(\tau^\pm_j)-A(u))D_\pm A(u_j), \\
    D_\pm \abseps{A(u_j)-A(u)} &= \signe(A(\theta^\pm_j)-A(u))D_\pm
    A(u_j).
  \end{align*}
  If $u$ is a differentiable function of $y$ then for each $j \in \Z$,
  \begin{equation}\label{chainRuleTrick}
    \signe'(A(\tau^\pm_j)-A(u))A(u)_{y} = -(\signe(A(\theta^\pm_j)-A(u)))_{y}.
  \end{equation}
  Both $\seq{\tau^\pm_j}_{j \in \Z}$ and $\seq{\theta^\pm_j}_{j \in
    \Z}$ depend on $u$ and $\varepsilon$.
\end{lemma}

\begin{proof}
  The first statement is a direct consequence of the mean value
  theorem.  Consider \eqref{chainRuleTrick}. First note that $\taum_j
  = \taup_{j-1}$ and $\thetam_j = \thetap_{j-1}$, so it suffices to
  consider $\taup_j$ and $\thetap_j$. If $u_j = u_{j+1}$ then
  $\theta_j = \tau_j$ is independent of $u$ and hence of $y$, so
  \eqref{chainRuleTrick} follows by the chain rule. In general,
  \begin{equation*}
    \begin{split}
      \signe'(A(\tau_j)-A(u))A(u)_{y}\Dp A(u_j)
      &= \Dp\signe(A(u_j)-A(u))A(u)_y \\
      &= -\Dp(\abseps{A(u_j)-A(u)})_y  \\
      &= -\signe(A(\theta_j)-A(u))_{y}\Dp A(u_j).
    \end{split}
  \end{equation*}
  In the case $u_j \neq u_{j+1}$ we have $\Dp A(u_j) \neq 0$ and
  \eqref{chainRuleTrick} follows.
\end{proof}

The following result is concerned with the pointwise limit of
$\signe(A(\theta^\pm_j)-A(u))$ as $\varepsilon \downto 0$.  The
explicit formula for this limit, which will be used later, shows that
the limit is in fact a Lipschitz continuous function in the case that
$A(u_j) \neq A(u_{j \pm 1})$.
\begin{lemma}\label{lemma:LimitOfSigne}
  Let
  \begin{equation*}
    \sg{(a,b)}{\sigma} :=
    \begin{cases} 
      \frac{|a-\sigma|-|b-\sigma|}{b-a} &\text{if $a \ne b$,}\\
      \sign(a-\sigma) & \text{if $a=b$, $\sigma \ne b$,}\\
      0 & \text{if $a=b=\sigma$},
    \end{cases}
  \end{equation*}
  for any real numbers $a$ and $b$. Under the same assumptions as in
  Lemma~\ref{DifQuotChainRule},
  \begin{equation*}
    \lim_{\eps\downto 0}\signe(A(\theta^\pm_j)-A(u)) 
    = -\sg{(A(u_j),A(u_{j \pm 1}))}{A(u)}.
  \end{equation*}
  Furthermore, if $a \neq b$ then
  \begin{equation*}
    \sg{(a,b)}{\sigma} = 
    \begin{cases}
      -1 & \text{if  $\sigma \le \min \seq{a,b}$,}\\
      \frac{2}{|a-b|}\left(\sigma-\frac{1}{2}(b+a)\right)
      & \text{if $\sigma \in \interval(a,b)$}, \\
      1& \text{if $\sigma\ge \max \seq{a,b}$.}
    \end{cases}
  \end{equation*}
\end{lemma}

\begin{proof}
  To prove the first statement we consider the case of
  $\thetap_j$. The same argument applies to $\thetam_j$.  Recall the
  definition of $\thetap_j$:
  \begin{equation*}
    \signe\left(A(\thetap_j)-A(u)\right)
    \left(A(u_{j+1})-A(u_j)\right) 
    = \abseps{A(u_{j+1})-A(u)}-\abseps{A(u_j)-A(u)}.
  \end{equation*}
  If $u_{j+1}=u_j$, then $\thetap_j=u_j$ for all $u$ and $\eps$, since
  $\thetap_j\in\interval(u_j,u_{j+1})$. Thus in this case
  \begin{equation*}
    \lim_{\eps\downto 0}\signe\left(A(\thetap_j)-A(u)\right)
    =
    \begin{cases}
      0 &\text{if $u=u_j$,}\\
      \sgn{A(u_j)-A(u)} &\text{otherwise.}
    \end{cases}
  \end{equation*}
  Now assume that $\Dp A(u_j)\ne 0$. Then
$$
\signe(A(\thetap_j)-A(u))
=\frac{\abseps{A(u_{j+1})-A(u)}-\abseps{A(u_{j})-A(u)}}{A(u_{j+1})-A(u_j)},
$$
and the result follows by letting $\eps\downto 0$.  Let us prove the
second statement. First observe that all expressions are symmetric in
$a$ and $b$, so we may assume that $a < b$. Under this assumption we
have
\begin{align*}
  (b-a)\sg{(a,b)}{\sigma} &= \abs{a-\sigma}-\abs{b-\sigma} \\
  &=\sgn{a-\sigma}(a-\sigma) - \sgn{b-\sigma}(b-\sigma) \\ & =
  \begin{cases}
    \sgn{b-\sigma}(a-b) &\text{if $\sigma \not \in (a,b)$,}\\
    2\sigma-(b+a) &\text{if $\sigma\in (a,b)$.}
  \end{cases}
\end{align*}
Dividing by $(b-a)$ concludes the proof.
\end{proof}

\begin{lemma}\label{lemma:LowerBoundSquareTerm}
  Let $\Ee_\Dx$ and $\test$ be defined in
  Lemma~\ref{lemma:MainLemmaTestFunc}.  For each $(x,t) \in \Pi_T$,
  \begin{align*}
    &\liminf_{\varepsilon \downto 0} \int_{\Pi_T} \Ee_\Dx\test\dyds \\
    & \quad \ge \int_{\Pi_T} \Dm\Biggl(\Dp \sign(A(u_\Dx)-A(u)) \\ &
    \qquad\qquad\qquad\qquad\quad\quad \times
    \left[\frac{1}{2}(A(u_\Dx)+A(S_\Dx
      u_\Dx))-A(u)\right]\Biggr)\test\dyds \\ &\quad \qquad +
    \liminf_{\varepsilon \downto 0} \frac{1}{2}\int_{\Pi_T}
    \left(\zetae(u_\Dx,\taum_\Dx,u) + \zetae(u_\Dx,\taup_\Dx,u)
    \right)(A(u)_y)^2 \test \dyds,
  \end{align*}
  where
$$
\zetae(a,b,c) := \signe'(A(a)-A(c))-\signe'(A(b)-A(c)), \qquad \forall
a,b,c\in \R.
$$
\end{lemma}

\begin{proof} 
  We split the proof into two claims.\\
  \emph{Claim 1.}
  \begin{align*}
    \Ee_j &\ge \frac{1}{2(\Dx)^2}\int^{u_j}_{u_{j+1}}
    \zetae(z,\taup_j,u)\partial_z(A(z)-A(u_{j+1}))^2\dz \\
    & \qquad + \frac{1}{2(\Dx)^2}\int^{u_j}_{u_{j-1}}
    \zetae(z,\taum_j,u)\partial_z(A(z)-A(u_{j-1}))^2\dz \\
    & \qquad + \frac{1}{2}\left(\zetae(u_j,\taum_j,u) +
      \zetae(u_j,\taup_j,u)\right)(A(u)_y)^2.
  \end{align*}
  \emph{Proof of Claim 1.}  Let
  \begin{align*}
    T^-  & :=    \frac{1}{(\Dx)^2}\int^{u_j}_{u_{j-1}}
    \psie''(z,u)(A(z)-A(u_{j-1}))\dz \\ 
    & \qquad - \frac{1}{\Dx}\int_{u_{j-1}}^{u_j}\psie''(z,u)\dz A(u)_y
    + \frac{1}{2}\partial_y\psie'(u,u_j)A(u)_y.
  \end{align*}
  We start by rewriting the first term as follows:
  \begin{align*}
    \frac{1}{(\Dx)^2}&\int^{u_j}_{u_{j-1}}
    \psie''(z,u)(A(z)-A(u_{j-1}))\dz \\
    & =  \frac{1}{2(\Dx)^2}\int^{u_j}_{u_{j-1}}
    \signe'(A(z)-A(u))\partial_z(A(z)-A(u_{j-1}))^2\dz \\ 
    & =  \frac{1}{2(\Dx)^2}\int^{u_j}_{u_{j-1}}
    \signe'(A(\taum_j)-A(u))\partial_z(A(z)-A(u_{j-1}))^2\dz \\ 
    & \qquad +  \frac{1}{2(\Dx)^2}\int^{u_j}_{u_{j-1}}
    \zetae(z,\taum_j,u)\partial_z(A(z)-A(u_{j-1}))^2\dz \\ 
    & = \frac{1}{2}\signe'(A(\taum_j)-A(u))(\Dm A(u_j))^2 \\
    & \qquad +\frac{1}{2(\Dx)^2}\int^{u_j}_{u_{j-1}}
    \zetae(z,\taum_j,u)\partial_z(A(z)-A(u_{j-1}))^2\dz.
  \end{align*}
  Concerning the second term in the definition of $T^-$,
  Lemma~\ref{DifQuotChainRule} gives
  \begin{align*}
    - \frac{1}{\Dx}\int_{u_{j-1}}^{u_j}
    \psie''(z,u)\dz A(u)_y &= -\Dm \signe(A(u_j)-A(u))A(u)_y \\
    & = -\signe'(A(\taum_j)-A(u))\Dm A(u_j)A(u)_y.
  \end{align*}
  For the last term we simply add and subtract to obtain
  \begin{align*}
    \frac{1}{2}\partial_y\psie'(u,u_j)A(u)_y & =
    \frac{1}{2}\signe'(A(\taum_j)-A(u))(A(u)_y)^2 \\ & \qquad +
    \frac{1}{2}\zetae(u_j,\taum_j,u)(A(u)_y)^2.
  \end{align*}
  Hence
  \begin{align*}
    T^- & = \frac{1}{2}\signe'(A(\taum_j)-A(u))(\Dm A(u_j)-A(u)_y)^2 \\
    & \qquad + \frac{1}{2(\Dx)^2} \int^{u_j}_{u_{j-1}}
    \zetae(z,\taum_j,u)\partial_z(A(z)-A(u_{j-1}))^2\dz \\
    &\qquad + \frac{1}{2}\zetae(u_j,\taum_j,u)(A(u)_y)^2.
  \end{align*}
  Define
  \begin{align*}
    T^+ := & \frac{1}{(\Dx)^2}\int^{u_j}_{u_{j+1}}
    \psie''(z,u)(A(z)-A(u_{j+1}))\dz \\ 
    & - \frac{1}{\Dx}\int_{u_{j+1}}^{u_j}\psie''(z,u)\dz A(u)_y +
    \frac{1}{2}\partial_y\psie'(u,u_j)A(u)_y.
  \end{align*}
  Using the same strategy as above we arrive at
  \begin{align*}
    T^+ &= \frac{1}{2}\signe'(A(\taup_j)-A(u))\left(\Dp
      A(u_j)-A(u)_y\right)^2 \\ &\qquad +
    \frac{1}{2(\Dx)^2}\int^{u_j}_{u_{j+1}}
    \zetae(z,\taup_j,u)\partial_z(A(z)-A(u_{j+1}))^2\dz \\ &\qquad +
    \frac{1}{2}\zetae(u_j,\taup_j,u)(A(u)_y)^2.
  \end{align*}
  Note that $\Ee_j = T^- + T^+$, so Claim 1 follows by removing the
  non-negative terms on the right hand side.\\* 
  \emph{Claim 2.} Suppose that $x \in I_j$. Then
  \begin{equation}\label{eq:Claim2LowerBoundLemma}
    \begin{split}
      & \liminf_{\varepsilon \downto 0} \frac{1}{2(\Dx)^2}\int_{\Pi_T}
      \Biggl[ \int^{u_j}_{u_{j-1}} \zetae(z,\taum_j,u)
      \frac{d}{dz}(A(z)-A(u_{j-1}))^2\dz \\ &
      \qquad\qquad\qquad\qquad\qquad\qquad +\int^{u_j}_{u_{j+1}}
      \zetae(z,\taup_j,u) \frac{d}{dz}(A(z)-A(u_{j+1}))^2\dz\Biggr]
      \test \dyds \\ & \qquad
      =\int_{\Pi_T}\Dm\Biggl(\Dp\sign(A(u_j)-A(u))
      \biggl[\frac{1}{2}(A(u_j)+A(u_{j+1}))-A(u)\biggr]\Biggr)\test
      \dyds.
    \end{split}
  \end{equation}
  \emph{Proof of Claim 2.}  Let
  \begin{align*}
    T_1^\varepsilon & :=
    \frac{1}{2(\Dx)^2}\int^{u_j}_{u_{j-1}}\signe'(A(z)-A(u))
    \frac{d}{dz} (A(z)-A(u_{j-1}))^2\dz, \\ K_1^\varepsilon &:=
    \frac{1}{2(\Dx)^2}\int^{u_j}_{u_{j-1}}\signe'(A(\taum_j)-A(u))
    \frac{d}{dz}(A(z)-A(u_{j-1}))^2\dz, \\
    T_2^\varepsilon &:=
    \frac{1}{2(\Dx)^2}\int^{u_j}_{u_{j+1}}\signe'(A(z)-A(u))
    \frac{d}{dz}(A(z)-A(u_{j+1}))^2\dz, \\
    K_2^\varepsilon &:=
    \frac{1}{2(\Dx)^2}\int^{u_j}_{u_{j+1}}\signe'(A(\taup_j)-A(u))
    \frac{d}{dz}(A(z)-A(u_{j+1}))^2\dz,
  \end{align*}
  and note that the left-hand side of \eqref{eq:Claim2LowerBoundLemma}
  may be written
$$
\liminf_{\varepsilon \downto 0}\int_{\Pi_T} \Bigl(
\left(T_1^\varepsilon - K_1^\varepsilon\right) + \left(T_2^\varepsilon
  - K_2^\varepsilon\right)\Bigr)\test\dyds.
$$
Let us rewrite $T_1^\varepsilon$ as follows:
\begin{align*}
  T_1^\varepsilon & = \frac{1}{(\Dx)^2}\int^{u_j}_{u_{j-1}}
  \signe'(A(z)-A(u))A'(z)(A(z)-A(u_{j-1}))\dz \\ & =
  \frac{1}{(\Dx)^2}\int^{u_j}_{u_{j-1}}
  \signe'(A(z)-A(u))A'(z)(A(u)-A(u_{j-1}))\dz \\ & \qquad +
  \frac{1}{(\Dx)^2}\int^{u_j}_{u_{j-1}}
  \signe'(A(z)-A(u))A'(z)(A(z)-A(u))\dz \\ & = \Dm
  \signe(A(u_j)-A(u))\frac{(A(u)-A(u_{j-1}))}{\Dx} +R^\varepsilon_1,
\end{align*}
where
\begin{align*}
  R^\varepsilon_1 & :=
  \frac{1}{\Dx^2}\Biggl[\signe(A(z)-A(u))(A(z)-A(u))\bigg|_{z=u_{j-1}}^{z=u_j}
  \\ & \qquad\qquad\qquad\qquad - \int^{u_j}_{u_{j-1}} \frac{d}{dz}
  \abseps{A(z)-A(u)}\dz \Biggr].
\end{align*}
Concerning $K_1^\varepsilon$, we apply Lemma~\ref{DifQuotChainRule} to
obtain
\begin{align*}
  K_1^\varepsilon & = \frac{1}{2(\Dx)^2}\int^{u_j}_{u_{j-1}}
  \signe'(A(\taum_j)-A(u))\partial_z(A(z)-A(u_{j-1}))^2\dz \\ & =
  \frac{1}{2}\signe'(A(\taum_j)-A(u))(\Dm A(u_j))^2\dz \\ & =
  \frac{1}{2}\Dm \signe(A(u_j)-A(u))\Dm A(u_j).
\end{align*}
It now follows that
$$
T_1^\varepsilon-K_1^\varepsilon =
-\frac{1}{\Dx}\Dm\signe(A(u_j)-A(u))\left[\frac{1}{2}\left(A(u_j)+A(u_{j-1})\right)-A(u)\right]
+ R^\varepsilon_1.
$$

Performing the same type of computations as above yields
$$
T_2^\varepsilon-K_2^\varepsilon = \frac{1}{\Dx}\Dp\signe(A(u_j)-A(u))
\left[\frac{1}{2}\left(A(u_{j+1})+A(u_{j})\right)-A(u)\right] +
R^\varepsilon_2,
$$
where
\begin{align*}
  R^\varepsilon_2 & := \frac{1}{\Dx^2}\Biggl[\int_{u_j}^{u_{j+1}}
  \frac{d}{dz}\abseps{A(z)-A(u)}\dz \\ & \qquad\qquad\qquad\qquad
  -\signe(A(z)-A(u))(A(z)-A(u))\bigg|^{z=u_{j+1}}_{z=u_j}\Biggr].
\end{align*}
Next, observe that
\begin{displaymath}
  \begin{split}
    R^\varepsilon_1 &=
    \frac{1}{\Dx^2}\Bigg[_{z=u_{j-1}}^{z=u_j}{}\signe(A(z)-A(u))(A(z)-A(u))-
    \abseps{A(z)-A(u)}\Bigg],
  \end{split}
\end{displaymath}
so $\lim_{\varepsilon \downto 0}R^\varepsilon_1 = 0$. The same
considerations apply to $R^\varepsilon_2$ so 
$\lim_{\varepsilon \downto 0}R^\varepsilon_2 = 0$ also. 
Claim 2 follows from an application of the dominated convergence 
theorem. Finally, combining Claim 1 and Claim 2 finishes the proof.
\end{proof}

\subsection{Estimates}\label{sec:Estimates}
The purpose of this section is to find bounds on the ``unwanted''
terms in inequality \eqref{MainIneq2} and
Lemma~\ref{lemma:LowerBoundSquareTerm}.  Throughout this section the
notation is the one given in Lemma~\ref{lemma:MainLemmaTestFunc}.  We
let $C$ denote a generic constant. By constant it is meant that it
does not depend on the ``small" variables but it might depend on $T$
and the initial data.  For any set $A$, let $\mathds{1}_A$ denote its
characteristic function.

For future reference we collect some standard estimates in a lemma.
\begin{lemma}\label{omegaLemma}
  Let $\test$ be defined in Lemma~\ref{lemma:MainLemmaTestFunc}. Then
  \begin{equation*}
    \abs{\frac{\partial^k}{\partial x^k}\test(x,t,y,s)} \le \psi(t)
    \frac{\norm{\rho^{(k)}}_{L^\infty}}{r^{k+1}}\car{|x-y|\le r}(x,y)\rho_{r_0}(t-s).
  \end{equation*}
  Recall that $S_\sigma \test(x,t,y,s) = \test(x+\sigma,t,y,s)$.  If
  $|\sigma| \le \Dx$ then
  \begin{equation*}
    \abs{\frac{\partial^k}{\partial x^k}S_\sigma \test(x,t,y,s)}
    \le \psi(t)\frac{\norm{\rho^{(k)}}_{L^\infty}}{r^{k+1}}\car{|x-y|
      \le r+\Dx}(x,y)\rho_{r_0}(t-s).
  \end{equation*}
  Considering the difference quotient applied to $\omega_r$ we have
  \begin{equation*}
    \abs{\Dp\omega_r(x-y)} 
    \le \frac{\norm{\rho'}_{L^\infty}}{r^2}\car{|x-y|\le r + \Dx}(x,y).
  \end{equation*}
\end{lemma}

\begin{proof}
  Note that
  \begin{equation*}
    \frac{\partial^k}{\partial x^k}\omega_r(x) = \frac{1}{r^{k+1}}\rho^{(k)}\left(\frac{x}{r}\right).
  \end{equation*}
  Since $\mathrm{supp}(\rho) \subset [-1,1]$ we have
  \begin{equation*}
    \abs{\frac{\partial^k}{\partial x^k}\omega_r(x)} \le 
    \frac{\norm{\rho^{(k)}}_{L^\infty}}{r^{k+1}}\car{|x|\le r}(x),
  \end{equation*}
  which proves the first statement.

  Consider the second statement.  If $|x-y| \ge r + \Dx$, then
  \begin{equation*}
    |x+\sigma-y| \ge |x-y|-|\sigma| \ge r+\Dx-\Dx = r,
  \end{equation*}
  so it follows that $\car{|x + \sigma -y| \le r}(x,y) \le
  \car{|x-y|\le r + \Dx}(x,y)$; this proves the second statement.

  To prove the last statement, recall that
$$
\Dp\omega_r(x) = \frac{\omega_r(x + \Dx)-\omega_r(x)}{\Dx}.
$$
If $|x| \ge r + \Dx$ then $\omega_r(x + \Dx)=\omega_r(x)=0$, so
$\mathrm{supp}(\Dp(\omega_r)) \subset [-r-\Dx,r+\Dx]$.  By the mean
value theorem and the fact that
$\norm{\omega_r'}_{L^\infty}=\norm{\rho'}_{L^\infty}r^{-2}$ we get
$$
\abs{\omega_r(x+\Dx)-\omega_r(x)} \le
\frac{\norm{\rho'}_{L^\infty}}{r^2}\Dx.
$$
The last statement follows from this.
\end{proof}

\begin{estimate}\label{Beta2Est}
  \begin{equation*}
    \left|\intPi \sgn{u_\Dx-u}\left(f(u_\Dx)-f(u)\right)
      \left(\Dp\test +\test_{y}\right) \dX \right| 
    \le C\frac{\Dx}{r}\left(1 + \frac{\Dx}{r}\right).
  \end{equation*}
\end{estimate}

\begin{proof}
  Let
$$
\beta := \intPi \sgn{u_\Dx-u}\left(f(u_\Dx)-f(u)\right) \left(\Dp\test
  +\test_{y}\right) \dX.
$$
First note that
$$
\Dp\test + \test_y = \Dp\test -\test_x.
$$
We claim that
\begin{equation}\label{DiffQuotDerDiff}
  \left(\Dp\test-\test_{x}\right)(x,t,y,s) 
  = \frac{1}{\Dx} \int_0^{\Dx} 
  (\Dx-\sigma)\test_{xx}(x + \sigma,t,y,s)\,d\sigma.
\end{equation}
Hence
$$
\beta = \frac{1}{\Dx} \intPi \int_0^{\Dx}\!\!
\signe\left(A(u_\Dx)-A(u)\right)\left(f(u_\Dx)-f(u)\right)
(\Dx-\sigma)S_\sigma\test_{xx}\,d\sigma \dX.
$$
We can write
\begin{align*}
  &\sgn{u_\Dx-u}\left(f(u_\Dx)-f(u)\right)(x,t,y,s) \\
  & \qquad = \sum_j
  \underbrace{\sgn{u_j-u}\left(f(u_j)-f(u)\right)(t,y,s)}_{\Theta_j}\car{I_j}(x).
\end{align*}
Using summation by parts
\begin{align*}
  \frac{1}{\Dx}\int_{\R}\int_0^{\Dx}
  &\sgn{u_\Dx-u}\left(f(u_\Dx)-f(u)\right)(\Dx-\sigma)
  S_\sigma\test_{xx} \, d\sigma dx \\ &= \frac{1}{\Dx}\int_0^{\Dx}
  \sum_j\Theta_j\int_{\R}\car{I_j}(x) (\Dx-\sigma)S_\sigma\test_{xx}\,
  dx d\sigma \\ &=\frac{1}{\Dx}\int_0^\Dx \sum_j \Theta_j \int_{I_j}
  \test_{xx}(x+\sigma,t,y,s) \,dx (\Dx-\sigma)\,d\sigma \\
  &=\int_0^\Dx \sum_j \Theta_j \left(\Dm
    S_\sigma\test_{x,j+1/2}\right) (\Dx-\sigma)\,d\sigma \\ &= -
  \sum_j \Dp\Theta_j \int_0^\Dx S_\sigma\test_{x,j+1/2}
  (\Dx-\sigma)\,d\sigma,
\end{align*}
where $S_\sigma\test_{x,j+1/2}=\test_x(x_{j+1/2}+\sigma,t,y,s)$.  By
Lemma~\ref{omegaLemma} we have
$$
\abs{\test_{x}(x+\sigma,t,y,s)} \le C\frac{1}{r^2}\car{|x-y|\le r +
  \Dx}(x,y)\rho_{r_0}(t-s).
$$
Hence
$$
\Bigl| \int_0^{\Dx} S_\sigma\test_{x,j+1/2}(\Dx-\sigma)\,
d\sigma\Bigr| \le C \frac{\Dx^2}{r^2}\car{|x_{j+1/2}-y|\le r +
  \Dx}(y)\rho_{r_0}(t-s).
$$
Now
$$
\abs{\Dp\Theta_j}\le \norm{f}_{\mathrm{Lip}} \abs{\Dp u_j}.
$$
Therefore
\begin{align*}
  &\Bigl| \frac{1}{\Dx}\int_{\R}\int_0^{\Dx}
  \sgn{u_\Dx-u}\left(f(u_\Dx)-f(u)\right)(\Dx-\sigma)
  \test_{xx}^\sigma\, d\sigma dx\Bigr| \\
  &\qquad \le \sum_j \abs{\Dp \Theta_j} \Bigl|
  \int_0^\Dx S_\sigma\test_{x,j+1/2} (\Dx-\sigma)\,d\sigma \Bigr| \\
  &\qquad \le C\norm{f}_{\mathrm{Lip}} \sum_j\abs{\Dp u_j}
  \frac{\Dx^2}{r^2}\car{|x_{j+1/2}-y|\le r + \Dx}(y)\rho_{r_0}(t-s).
\end{align*}
It follows by the above and Lemma~\ref{SemiDiscProp} that
\begin{align*}
  \abs{\beta} & \le C \Dx^2\frac{r+\Dx}{r^2}\int_0^T\sum_j\abs{\Dp
    u_j}\, dt \\ &= C\frac{r+\Dx}{r^2}\int_{\Pi_T}
  \abs{u_\Dx(x+\Dx,t)-u_\Dx(x,t)} \,dxdt \\
  &= CT\frac{1}{r}\left(1+\frac{\Dx}{r}\right)
  \Dx\abs{u^0_\Dx}_{BV(\R)}.
\end{align*}
This concludes the proof.
\end{proof}

\begin{estimate}\label{DoubleDerEst}
  \begin{align*}
    &\bigg|\intPi \abs{A(u_\Dx)-A(u)} \left(\Dm\Dp\test + (\Dp +
      \Dm)\test_{y} + \test_{yy}\right) \dX \bigg| \\ & \qquad \qquad
    \le C \frac{\Dx^2}{r^3}\left(1 + \frac{\Dx}{r}\right).
  \end{align*}
\end{estimate}

\begin{proof}
  Since $\test_{xx} + 2\test_{xy} + \test_{yy} = 0$ it follows that
$$
\Dm\Dp\test + (\Dp + \Dm)\test_{y} + \test_{yy} =
\left(\Dm\Dp\test-\test_{xx}\right) +\left((\Dp + \Dm)\test -
  2\test_{x}\right)_{y}.
$$
Thus
\begin{align*}
  \intPi & \abs{A(u_\Dx)-A(u)} \left(\Dm\Dp\test + (\Dp +
    \Dm)\test_{y} + \test_{yy}\right) \dX \\ &=\intPi
  \abs{A(u_\Dx)-A(u)}\left(\Dm\Dp\test-\test_{xx}\right) \dX \\
  &\qquad + \intPi \abs{A(u_\Dx)-A(u)} \left((\Dp +
    \Dm)\test-2\test_{x}\right)_{y} \dX \\ &=: \zeta_1 + \zeta_2.
\end{align*}
Consider the term $\zeta_1$. We use the same strategy as in
Estimate~\ref{Beta2Est}.  Writing $\mu(\sigma)=\test(x+\sigma,t,y,s)$,
a Taylor expansion gives
\begin{equation*}
  \mu(z)-\mu(0) = z\mu'(0)+\frac{1}{2}z^2\mu''(0)
  +\frac{1}{6}z^3\mu^{(3)}(0)
  -\frac{1}{6}\int_0^z(\sigma-z)^3\mu^{(4)}(\sigma)\,d\sigma.  
\end{equation*}
Using this, we get
\begin{align*}
  & \mu(\Dx)-2\mu(0)+\mu(-\Dx) - \Dx^2\mu''(0) \\ & \qquad
  =-\frac{1}{6}\int_0^{\Dx}(\sigma-\Dx)^3\mu^{(4)}(\sigma)\,d\sigma +
  \frac{1}{6}\int_{-\Dx}^{0}(\sigma+\Dx)^3\mu^{(4)}(\sigma)\,d\sigma.
\end{align*}
It follows that
\begin{align*}
  \Dp\Dm\test - \test_{xx} & =-\frac{1}{6\Dx^2}
  \int_0^{\Dx}(\sigma-\Dx)^3 \frac{\partial^4}{\partial
    x^4}\test(x+\sigma,t,y,s)\,d\sigma \\ & \qquad +
  \frac{1}{6\Dx^2}\int_{-\Dx}^{0}(\sigma+\Dx)^3
  \frac{\partial^4}{\partial x^4}\test(x+\sigma,t,y,s)\,d\sigma.
\end{align*}
Splitting $\zeta_1$ according to this equality we get
\begin{align*}
  \zeta_1 &=\intPi
  \abs{A(u_\Dx)-A(u)}\left(\Dm\Dp\test-\test_{xx}\right) \dX \\ &=
  -\frac{1}{6\Dx^2}\intPi \int_0^{\Dx}
  \abs{A(u_\Dx)-A(u)}(\sigma-\Dx)^3\frac{\partial^4}{\partial x^4}
  \test(x+\sigma,t,y,s)\,d\sigma\dX \\ &\qquad +
  \frac{1}{6\Dx^2}\intPi \int_{-\Dx}^{0}
  \abs{A(u_\Dx)-A(u)}(\sigma+\Dx)^3\frac{\partial^4}{\partial x^4}
  \test(x+\sigma,t,y,s) \,d\sigma\dX \\
  &=: \zeta_{1,1} + \zeta_{1,2} .
\end{align*}
We also have that
$$
\abs{A(u_\Dx)-A(u)}(x,t,y,s) = \sum_j
\underbrace{\abs{A(u_j)-A(u)}(t,y,s)}_{\Phi_j}\car{I_j}(x).
$$
Now consider $\zeta_{1,1}$,
\begin{align*}
  -\int_0^{\Dx}&\int_{\R}\abs{A(u_\Dx)-A(u)}
  (\sigma-\Dx)^3\frac{\partial^4}{\partial x^4}S_\sigma\test \,dx
  d\sigma \\ &=-\sum_j
  \abs{A(u_j)-A(u)}(t,y,s)\int_0^{\Dx}(\sigma-\Dx)^3\int_{\R}
  \car{I_j}(x)\frac{\partial^4}{\partial x^4}S_\sigma \test \,dx
  d\sigma \\ &= -\Dx\int_0^{\Dx}(\sigma-\Dx)^3 \sum_j \Phi_j \Dm
  \test^{\sigma}_{xxx,j+1/2} \, d\sigma \\
  &=\Dx\sum_j\Dp\Phi_j\int_0^{\Dx}(\sigma-\Dx)^3
  S_\sigma\test_{xxx,j+1/2} \, d\sigma,
\end{align*}
where
$$
S_\sigma\test_{xxx,j+1/2}(t,y,s) = \frac{\partial^3}{\partial
  x^3}\test(x_{j+1/2}+\sigma,t,y,s).
$$
Now we use Lemma~\ref{omegaLemma} to estimate this term as follows:
\begin{align*}
  \abs{\zeta_{1,1}}&= \Bigl|\frac{1}{6\Dx^2} \intPi
  \int_0^{\Dx}\abs{A(u_\Dx)-A(u)}
  (\sigma-\Dx)^3\frac{\partial^4}{\partial x^4}
  \test(x+\sigma,t,y,s)\,d\sigma \dX\Bigr|
  \\
  &= \Bigl| \frac{1}{6\Dx} \int_{\Pi_T} \int_0^T \sum_j
  \Dp\Phi_j\int_0^{\Dx}(\sigma-\Dx)^3 S_\sigma\test_{xxx,j+1/2}
  \,d\sigma\, \,dt\dyds\Bigr|
  \\
  &\le C\frac{r+\Dx}{\Dx^2 r^4} \int_{\Pi_T} \abs{\Dp A(u_\Dx)}
  \Bigl(\int_0^{\Dx}(\sigma-\Dx)^3\,d\sigma\Bigr)\,dxdt \\ & \le
  C\Dx^2\frac{r+\Dx}{r^4} =
  C\frac{\Dx^2}{r^3}\left(1+\frac{\Dx}{r}\right),
\end{align*}
where we have used that $\abs{A(u_\Dx(\cdot,t))}_{BV(\R)}$ is bounded
independently of $\Delta x,t, \eta$ by Lemma~\ref{SemiDiscProp}.  The
term $\zeta_{1,2}$ is estimated in a similar way.

Now consider $\zeta_2$. Again, let $\mu(\sigma) =
\test(x+\sigma,t,y,s)$. Then
$$
(\Dp + \Dm)\test -2\test_x = \frac{1}{\Dx}
\left[\mu(\Dx)-\mu(-\Dx)-2\Dx\mu'(0)\right].
$$
By a Taylor expansion
$$
\mu(z)-\mu(0) = z\mu'(0) + \frac{1}{2}z^2 \mu''(0)
+\frac{1}{2}\int_0^z(\sigma-z)^2\mu^{(3)}(\sigma)\,d\sigma.
$$
Puting $z = \pm \Dx$ and subtracting the corresponding equations we
obtain
\begin{align*}
  (\Dp + \Dm)\test -2\test_x & =
  \frac{1}{2\Dx}\int_0^\Dx(\sigma-\Dx)^2
  \frac{\partial^3}{\partial x^3}\test(x+\sigma,t,y,s)\,d\sigma \\
  & \qquad +\frac{1}{2\Dx}\int^0_{-\Dx}(\sigma+\Dx)^2
  \frac{\partial^3}{\partial x^3}\test(x+\sigma,t,y,s)\,d\sigma.
\end{align*}
We may split $\zeta_2$ into the two terms
\begin{align*}
  \zeta_2 &= \frac{1}{2\Dx}\intPi\int_0^\Dx \abs{A(u_\Dx)-A(u)}
  (\sigma-\Dx)^2\frac{\partial^3}{\partial x^3}
  \frac{\partial}{\partial y}\test(x+\sigma,t,y,s)\,d\sigma\dX \\
  &\quad
  +\frac{1}{2\Dx}\intPi\int^0_{-\Dx}\abs{A(u_\Dx)-A(u)}(\sigma+\Dx)^2
  \frac{\partial^3}{\partial x^3}\frac{\partial}{\partial
    y}\test(x+\sigma,t,y,s)\,d\sigma \dX \\
  &=:\zeta_{2,1}+\zeta_{2,2}.
\end{align*}
Performing integration by parts, $\zeta_{2,1}$ becomes
$$
\frac{1}{2\Dx}\intPi\int_0^\Dx \sgn{A(u_\Dx)-A(u)}A(u)_y(\sigma-\Dx)^2
\frac{\partial^3}{\partial x^3}\test(x+\sigma,t,y,s)\,d\sigma\dX.
$$
Thus, by Lemma~\ref{omegaLemma},
\begin{align*}
  |\zeta_{2,1}| &\le \frac{1}{2\Dx}\intPi |A(u)_y|
  \abs{\int_0^\Dx(\sigma-\Dx)^2
    \frac{\partial^3}{\partial x^3}\test(x+\sigma,t,y,s)\,d\sigma}\dX \\
  & \le T C \frac{r + \Dx}{r^4\Dx}
  \left(\int_0^\Dx(\sigma-\Dx)^2\,d\sigma \right)
  \int_{\Pi_T}|A(u)_y|\dyds \\
  & \le C \frac{\Dx^2}{r^3}\left(1 + \frac{\Dx}{r}\right),
\end{align*}
as $\abs{A(u(\cdot,s))}_{BV(\R)} \le \abs{A(u^0(\cdot))}_{BV(\R)}$ for
all $s$.  The same estimate holds for $\zeta_{2,2}$.
\end{proof}

\begin{estimate} \label{gamma3Est}
  \begin{equation*}
    \left| \intPi \left(\int_{u_\Dx}^{S_\Dx u_{\Dx}} 
	\sgn{z-u}F_2'(z)\,dz\right) \,\Dp\test \dX\right| 
    \le C\frac{\Dx}{r}\left(1 + \frac{\Dx}{r}\right).
  \end{equation*}
\end{estimate}

\begin{proof}
  By definition $F_2'$ is bounded. Hence,
  \begin{equation*}
    \Bigl| \int_{u_j}^{u_{j+1}} \sgn{z-u} F_2'(z)\,dz\Bigr| 
    \le \norm{F_2}_{\mathrm{Lip}} \Dx \abs{\Dp u_j}.
  \end{equation*}
  Note that $\abs{u_\Dx(\cdot,t)}_{BV(\R)}$ is bounded independently
  of $\Dx,t,\eta$ by Lemma~\ref{SemiDiscProp}, so we may apply
  Lemma~\ref{omegaLemma} to obtain the result.
\end{proof}

Next, we consider the terms from
Lemma~\ref{lemma:LowerBoundSquareTerm}.
\begin{estimate}\label{estimate:DissTermSchemeErr}
  \begin{align*}
    & \intPi \Dm\left(\Dp \sign(A(u_j)-A(u))
      \left[\frac{1}{2}(A(u_j)+A(u_{j+1}))-A(u)\right]\right) \test \dX \\
    & \qquad \qquad \ge -C(1+r+\Delta
    x)\frac{\Dx}{r^2}\left(1+\frac{\Dx}{r}\right)^3.
  \end{align*}
\end{estimate}

\begin{proof}
  Let us first show that
  \begin{multline}\label{ineq:TransThetaEst}
    \abs{\Dp
      \sign(A(u_j)-A(u))\left[\frac{1}{2}(A(u_j)+A(u_{j+1}))-A(u)\right]}
    \\ \le \Dx \Dp \sign(A(u_j)-A(u))\Dp(A(u_j)).
  \end{multline}
  First note that
  \begin{align*}
    \Dp \sign(A(u_j)&-A(u)) \\
    & = \frac{2}{\Dx}\sign(A(u_j)-A(u_{j+1})) \car{A(u) \in \interval
      (A(u_j),A(u_{j+1})},
  \end{align*}
  so the left-hand side of \eqref{ineq:TransThetaEst} is zero whenever
  $A(u) \notin \interval (A(u_j),A(u_{j+1}))$.  Second, if $c \in
  \interval(a,b)$, then it follows that
$$
\abs{\frac{1}{2}(a + b)-c} = \frac{1}{2}(|b-c| + |a-c|) \le |b-a|.
$$
Since $z \mapsto \sgn{A(z)-A(u)}$ is increasing, the right-hand side
is positive. This proves \eqref{ineq:TransThetaEst}.

Performing integration by parts we obtain
\begin{align*}
  & \bigg|\intPi \Dm\left(\Dp \sign(A(u_j)-A(u))
    \left[\frac{1}{2}(A(u_j)+A(u_{j+1}))-A(u)\right]\right) \test \dX \bigg|\\
  & \qquad \le \intPi \abs{\Dp \sign(A(u_j)-A(u))
    \left[\frac{1}{2}(A(u_j)+A(u_{j+1}))-A(u)\right]}\abs{\Dp \test} \dX \\
  & \qquad \le \Dx \intPi \Dp \sign(A(u_j)-A(u))\Dp(A(u_j)) \abs{\Dp
    \test}\dX.
\end{align*}
Using integration by parts for difference quotients and the Leibniz
rule for difference quotients, we obtain
\begin{align*}
  & \intPi\Dp\signe(A(u_\Dx)-A(u))\Dp A(u_\Dx)\abs{\Dp\test} \dX \\ &
  \quad = -\intPi \signe(A(u_\Dx)-A(u))\Dp A(u_\Dx)\Dm\abs{\Dp\test}
  \dX \\ & \quad \qquad -\intPi \signe(A(u_\Dx)-A(u))\Dm\Dp
  A(u_\Dx)\abs{\Dm\test}\dX \\ & \quad =: \zeta_1+\zeta_2.
\end{align*}

To estimate $\zeta_1$ we first observe that $\Dm\abs{\Dp\test} \le
\abs{\Dp\Dm\test}$.  Furthermore, when proving
Estimate~\ref{DoubleDerEst}, we established that
\begin{align*}
  &\Dp\Dm\test(x,t,y,s) \\
  &\qquad = \test_{xx}(x,t,y,s) -\frac{1}{6\Dx^2}
  \int_0^{\Dx}(\sigma-\Dx)^3
  \frac{\partial^4}{\partial x^4}\test(x+\sigma,t,y,s)\,d\sigma \\
  &\qquad \qquad + \frac{1}{6\Dx^2}\int_{-\Dx}^{0}\!(\sigma+\Dx)^3
  \frac{\partial^4}{\partial x^4} \test(x+\sigma,t,y,s)\,d\sigma.
\end{align*}
By Lemma~\ref{omegaLemma},
\begin{align*}
  & \Bigl| \int_0^{\pm\Dx}(\sigma\mp\Dx)^3\frac{\partial^4}{\partial
    x^4} \test(x+\sigma,t,y,s)\,d\sigma\Bigr | \\ & \qquad \le
  C\frac{(\Dx)^4}{r^5}\car{|x-y|\le r+\Dx}(x,y)\rho_{r_0}(t-s).
\end{align*}
Using Lemma~\ref{omegaLemma} once more, the above implies that
\begin{align*}
  \int_{\Pi_T} \abs{\Dp\Dm\test}\dyds &\le \int_{\Pi_T}
  \abs{\test_{xx}}\dyds + C\frac{\Dx^2}{r^4}
  \left(1+\frac{\Dx}{r}\right) \\ &\le C\left(\frac{1}{r^2}
    +\frac{\Dx^2}{r^4}\right)\left(1 + \frac{\Dx}{r}\right).
\end{align*}
Therefore,
\begin{align*}
  \abs{\zeta_1}&= \Bigl| \intPi \signe(A(u_\Dx)-A(u))\Dp
  A(u_\Dx)\Dp\abs{\Dm\test} \dX\Bigr|
  \\
  & \le \int_{\Pi_T} \abs{\Dp A(u_\Dx)} \Bigl(\int_{\Pi_T}
  \abs{\Dp\Dm\test} \dyds\Bigr) \dxdt
  \\
  &\le C\left(\frac{1}{r^2} + \frac{\Dx^2}{r^4}\right)
  \left(1+\frac{\Dx}{r}\right) \int_{\Pi_T} \abs{\Dp A(u_\Dx)}\dxdt.
\end{align*}
Recall that $\abs{A(u_\Dx(\cdot,t))}_{BV(\R)}$ is bounded
independently of $\Delta x,t, \eta$ by Lemma~\ref{SemiDiscProp}.

Concerning $\zeta_2$ we have
\begin{align*}
  \abs{\zeta_2}& = \Bigl| \intPi \signe(A(u_\Dx)-A(u)) \left(\Dm\Dp
    A(u_\Dx)\right)\abs{\Dm\test}\dX\Bigr| \\ &\le \intPi \abs{\Dm\Dp
    A(u_\Dx)}\abs{\Dm\test}\dX \\ &\le C\frac{r + \Dx}{r^2}
  \int_{\Pi_T} \abs{\Dm\Dp A(u_\Dx)}\dxdt.
\end{align*}
Note that it follows from \eqref{SemiDiscEqu} and
Lemma~\ref{SemiDiscFluxDiffBounds} that $\norm{\Dm\Dp
  A(u_\Dx(\cdot,t))}_{L^1(\R)}$ is bounded independently of $\Dx,t,
\eta$. Hence,
\begin{align*}
  & \Dx\intPi \Dp\signe(A(u_\Dx)-A(u)) \Dp A(u_\Dx)\abs{\Dp\test}\dX
  \\ &\qquad \le \Dx\left(\abs{\zeta_1}+\abs{\zeta_2}\right) \\
  & \qquad \le C(1 + r + \Delta x)\left(\frac{\Dx}{r^2}
    +\frac{\Dx^3}{r^4}\right)\left(1 +\frac{\Delta x}{r}\right)\\
  & \qquad \le C(1+r+\Delta x)\frac{\Dx}{r^2}
  \left(1+\frac{\Dx}{r}\right)^3.
\end{align*}
\end{proof}

\begin{estimate}\label{est:Re}
  \begin{align*}
    & \int_{\Pi_T}\left(\liminf_{\varepsilon \downto 0}
      \int_{\Pi_T}\frac{1}{2}\left(\zetae(u_\Dx,\taum_\Dx,u) +
        \zetae(u_\Dx,\taup_\Dx,u)\right)(A(u)_y)^2\test\dyds\right)\dxdt
    \\ & \qquad\qquad \ge -C\left(\frac{\Dx}{r_0} + \frac{\Dx}{r} +
      \frac{\Dx}{r^2}\right)
  \end{align*}
\end{estimate}

\begin{proof}
  Set
  \begin{align*}
    & \mathcal{R}_{j}^\varepsilon := \left(\zetae(u_j,\taum_j,u) +
      \zetae(u_j,\taup_j,u)\right)(A(u)_y)^2 \\ & \qquad =
    (\signe'(A(u_j)-A(u))-\signe'(A(\taum_j)-A(u)))(A(u)_y)^2 \\ &
    \qquad \qquad
    +(\signe'(A(u_j)-A(u))-\signe'(A(\taup_j)-A(u)))(A(u)_y)^2,
  \end{align*}
  and $\mathcal{R}^\varepsilon_\Dx(x,t,y,s) =
  \mathcal{R}_j^\varepsilon(y,t,s)$ for $x \in I_j$. Note that the
  term we want to estimate may be written
  \begin{align*}
    & \liminf_{\varepsilon \downto 0}\int_{\Pi_T}
    \mathcal{R}^\varepsilon_\Dx(x,t,y,s)\test(x,t,y,s)\dyds \\ &
    \qquad = \sum_j\liminf_{\varepsilon \downto 0}
    \left(\int_{\Pi_T}\mathcal{R}^\varepsilon_j(t,y,s)
      \test(x,t,y,s)\dyds\right) \car{I_j}(x).
  \end{align*}
  Let us define an entropy function by
  \begin{align*}
    & \partial_u \Psi_\varepsilon(u,u_{j-1},u_j,u_{j+1}) \\ & \quad :=
    \signe(A(\thetam_j)-A(u)) -2\signe(A(u_j)-A(u)) +
    \signe(A(\thetap_j)-A(u)).
  \end{align*}
  Recall that $\theta_j^\pm = \theta_j^\pm(u)$, so the above function
  is not as explicit as it appears. However, by
  Lemma~\ref{lemma:LimitOfSigne} we are able to obtain an explicit
  expression for the limit as $\eps\to 0$.  To simplify the notation we write
  $\Psi_{\varepsilon,j}'(u)$ for $\partial_u
  \Psi_\varepsilon(u,u_{j-1},u_{j},u_{j+1})$.  Let us also define the
  entropy flux functions
\begin{equation*}
\Xi_{\varepsilon,j}'(u) = \Psi_{\varepsilon,j}'(u)f'(u), \qquad
\Phi_{\varepsilon,j}'(u) = \Psi_{\varepsilon,j}'(u)A'(u).
\end{equation*}
That is
$(\Psi_{\varepsilon,j},\Xi_{\varepsilon,j},\Phi_{\varepsilon,j})$ is
an entropy-entropy flux triple.

Multliplying equation \eqref{MainProblem} by
$\Psi_{\varepsilon,j}'(u)$ yields
\begin{equation*}
\Psi_{\varepsilon,j}(u)_s + \Xi_{\varepsilon,j}(u)_y =
\Phi_{\varepsilon,j}(u)_{yy} - \partial_y
\Psi_{\varepsilon,j}'(u)A(u)_y.
\end{equation*}
By Lemma~\ref{DifQuotChainRule} we see that
\begin{align}
  \partial_y\Psi_{\varepsilon,j}'(u)A(u)_y 
  &=\Bigl[\signe(A(\thetam_j)-A(u))_y -2\signe(A(u_j)-A(u))_y\notag\\
  &\hphantom{=-\Bigl[ }\quad +
    \signe(A(\thetap_j)-A(u))_y
  \Bigr]  A(u)_y\notag\\
  &=-\Bigl[\signe'(A(\taum_j)-A(u)) -2\signe'(A(u_j)-A(u))\notag\\
  &\hphantom{=-\Bigl[ }\quad +
    \signe'(A(\taup_j)-A(u))
  \Bigr]  \left(A(u)_y\right)^2\notag\\
  &=
  \mathcal{R}_j^\varepsilon.\label{eq:errRepr}
\end{align}
It follows that we can write
\begin{equation}\label{eq:DissipErrTermExpresion}
  \begin{split}
    \int_{\Pi_T}\mathcal{R}^\varepsilon_j\test \dyds &
    =\int_{\Pi_T}\Psi_{\varepsilon,j}(u)\test_s +
    \Xi_{\varepsilon,j}(u) \test_y +
    \Phi_{\varepsilon,j}(u)\test_{yy}\dyds \\ & =: T_1^\varepsilon +
    T_2^\varepsilon + T_3 ^\varepsilon.
  \end{split}
\end{equation}
Let us consider the three terms separately.

By Lemma~\ref{lemma:LimitOfSigne},
\begin{align*}
  \lim_{\varepsilon \downto 0}\Psi_{\varepsilon,j}(u) &=
  \lim_{\varepsilon \downto 0}\int_{u_j}^u \Psi_{\varepsilon,j}'(z)\dz
  = \int_{u_j}^u \lim_{\varepsilon \downto 0}
  \Psi_{\varepsilon,j}'(z)\dz
  \\ & = \int_{u_j}^u  \sign(A(z)-A(u_j))-\sg{j-1}{A(z)}\dz \\
  & \qquad + \int_{u_j}^u \sign(A(z)-A(u_j)) - \sg{j}{A(z)}\dz,
\end{align*}
where $\sg{j}{\sigma} := \sg{(A(u_j),A(u_{j+1}))}{\sigma}$.  Again by
Lemma~\ref{lemma:LimitOfSigne}, the mapping
$$
z \mapsto \sign(A(z)-A(u_j))-\sg{j-1}{A(z)}
$$
has support in $\interval(u_j,u_{j-1})$. Similar considerations apply
to the second term. Hence
\begin{align*}
  \abs{\lim_{\varepsilon \downto 0}\Psi_{\varepsilon,j}(u)}
  & \le \abs{\int_{u_j}^{u_{j-1}}  \sign(A(z)-A(u_j))-\sg{j-1}{A(z)}\dz} \\
  & \qquad + \abs{\int_{u_j}^{u_{j+1}}  \sign(A(z)-A(u_j)) - \sg{j}{A(z)}\dz} \\
  & \le 2\abs{u_j-u_{j-1}} + 2 \abs{u_{j+1}-u_{j}}.
\end{align*}
By the same type of reasoning we obtain the bound
\begin{align*}
  \abs{\lim_{\varepsilon\downto 0}\Xi_j^\varepsilon(u)} &\le
  \left|\int_{u_j}^{u_{j-1}}
    \left[\sign(A(z)-A(u_j))-\sg{j-1}{A(z)}\right]f'(z)\dz \right| \\
  &\qquad +\left|\int_{u_j}^{u_{j+1}} \left[\sign(A(z)-A(u_j)) -
      \sg{j}{A(z)}\right]f'(z)\dz\right| \\ &\le 2
  \norm{f'}_{L^\infty} \left(\abs{u_j-u_{j-1}} +
    \abs{u_{j+1}-u_j}\right).
\end{align*}
Concerning $\Phi_j^\varepsilon$ we use substitution and the
explicit expression given in Lemma~\ref{lemma:LimitOfSigne}. This
leads to
\begin{align*}
  \abs{\lim_{\varepsilon \downto 0}\Phi_j^\varepsilon(u)} & = \left|
    \int_{u_j}^u \left(-\sg{j-1}{A(z)} + 2\,\sign(A(z)-A(u_j)) -
      \sg{j}{A(z)}\right)A'(z)\dz\right| \\ &=
  \left|\int_{A(u_j)}^{A(u)} -\sg{j-1}{\sigma}
    + 2\,\sign(\sigma-A(u_j)) - \sg{j}{\sigma}\,d\sigma\right| \\
  & \le \left|\int_{A(u_j)}^{A(u_{j-1})} \sign(\sigma-A(u_j)) -
    \sg{j-1}{\sigma}\,d\sigma \right| \\ & \qquad +
  \left|\int_{A(u_j)}^{A(u_{j+1})} \sign(\sigma-A(u_j))-
    \sg{j}{\sigma}\,d\sigma\right| \\ & \le \abs{A(u_j)-A(u_{j-1})} +
  \abs{A(u_{j+1})-A(u_j)}.
\end{align*}

Let us return to equation \eqref{eq:DissipErrTermExpresion}.  By the
dominated convergence theorem and the above computations
\begin{align*}
  &\abs{\lim_{\varepsilon \downto 0}T_1^\varepsilon} \le
  \norm{\lim_{\varepsilon\downto 0}\Psi_j^\varepsilon}_{L^\infty}
  \int_{\Pi_T} |\test_s| \dyds \le C\left(\abs{\Dm u_j}+\abs{\Dp
      u_j}\right), \\ & \abs{\lim_{\varepsilon \downto
      0}T_2^\varepsilon} \le \norm{\lim_{\varepsilon \downto
      0}\Xi_j^\varepsilon}_{L^\infty} \int_{\Pi_T}|\test_y|\dyds
  \le C\frac{\Dx}{r}\left(\abs{\Dm u_j}+\abs{\Dp u_j}\right),\\
  &\abs{\lim_{\varepsilon \downto 0} T_3^\varepsilon} \le
  \norm{\lim_{\varepsilon \downto 0}\Phi_j^\varepsilon}_{L^\infty}
  \int_{\Pi_T}|\test_{yy}|\dyds \le C\frac{\Dx}{r^2}\left(\abs{\Dm
      A(u_j)}+\abs{\Dp A(u_j)}\right).
\end{align*}

Hence
\begin{align*}
  & \int_{\Pi_T}\sum_j\liminf_{\varepsilon \downto 0}
  \left(\int_{\Pi_T}\mathcal{R}^\varepsilon_j(t,y,s)\test(x,t,y,s)\dyds\right)
  \car{I_j}(x)\dxdt \\ & \qquad \ge -C\left(\frac{\Dx}{r_0} +
    \frac{\Dx}{r}\right) \int_{\Pi_T} \abs{\Dm u_\Dx}+\abs{\Dp u_\Dx}
  \dxdt \\ & \qquad \qquad -C\frac{\Dx}{r^2}\int_{\Pi_T}\abs{\Dm
    A(u_\Dx)}+\abs{\Dp A(u_\Dx)} \dxdt.
\end{align*}
The desired estimate now follows from the uniform bounds in 
Lemma~\ref{SemiDiscProp}.
\end{proof}

\subsection{Proof of Theorem~\ref{MainResultSemiDisc}}
Let us now combine the previous results to conclude the proof of
Theorem~\ref{MainResultSemiDisc}.  We begin by stating a rather
standard lemma.
\begin{lemma}\label{lem:kappa}
Set
$$
\kappa(t) := \int_{\R}\int_{\Pi_T}
\abs{u_\Dx(x,t)-u(y,s)}\omega_r(x-y)\rho_{r_0}(t-s) \,dydsdx.
$$
Let $t \ge r_0$, and denote by $L_c$ the Lipschitz constant of $t
\mapsto \norm{u(\cdot,t)}_{L^1(\R)}$. Then
$$
\abs{\kappa(t) - \norm{u_\Dx(\cdot,t)-u(\cdot,t)}_{L^1(\R)}} \le
\abs{u(\cdot,t)}_{BV(\R)}r + L_cr_0.
$$
\end{lemma}

\begin{proof}
  By the reverse triangle inequality,
  \begin{align*}
    &\abs{\kappa(t) - \norm{u_\Dx(\cdot,t)-u(\cdot,t)}_{L^1(\R)}} \\ &
    \qquad \le \!\int_\R\int_{\Pi_T}\!\!
    \abs{u(y,s)-u(x,t)}\omega_r(x-y)\rho_{r_0}(t-s)\dyds dx \\
    &\qquad \le \int_0^T \left(\int_{\R} |u(y,s)-u(t,y)|\,dy\right)
    \rho_{r_0}(t-s)\,ds \\
    &\qquad \qquad + \int_\R\int_\R|u(t,y)-u(x,t)|\omega_r(x-y)dydx \\
    & \qquad \le L_cr_0 + \abs{u(\cdot,t)}_{BV(\R)}r.
  \end{align*}
\end{proof}

\begin{proof}[Proof of Theorem~\ref{MainResultSemiDisc}]%\label{proof:SemidiscThm}
Our starting point is Lemma~\ref{lemma:MainLemmaTestFunc}.  
Let $A(\sigma) = \hat{A}(\sigma) + \eta \sigma$, where 
$\hat{A}$ is the original degenerate diffusion function. Let
  \begin{align*}
    \Xi &= \intPi \sgn{u_\Dx-u} \left(f(u_\Dx)-f(u)\right)
    \left(\Dp\test +\test_{y}\right)\dX \\ & \qquad + \intPi
    \left(\int_{u_\Dx}^{S_\Dx u_{\Dx}} \sgn{z-u}F_2'(z)\,dz\right)
    \,\Dp\test \dX \\ &\qquad + \intPi \abs{A(u_\Dx)-A(u)}
    \left(\Dm\Dp\test + (\Dp + \Dm)\test_{y} + \test_{yy}\right) \dX.
  \end{align*}
  By Estimate~\ref{Beta2Est}, Estimate~\ref{DoubleDerEst}, and
  Estimate~\ref{gamma3Est}, it follows that
  \begin{equation}\label{eq:E1Def}
    \abs{\Xi} \le C\frac{\Dx}{r}\left(1 + \frac{\Dx}{r^2}\right)\left(1 + \frac{\Dx}{r}\right) 
    =:E_1.
  \end{equation}

  Furthermore, by Lemma~\ref{lemma:FatousApplicationToDissip},
  Lemma~\ref{lemma:LowerBoundSquareTerm},
  Estimate~\ref{estimate:DissTermSchemeErr}, and
  Estimate~\ref{est:Re}, it follows that
  \begin{equation}\label{eq:E2def}
    \liminf_{\varepsilon \downto 0} \intPi \Ee_\Dx \test \dX 
    \ge -C(1 + r + \Dx)\frac{\Dx}{r^2}
    \left(1+\frac{\Dx}{r}\right)^3-C\frac{\Dx}{r_0} =:-E_2.
  \end{equation}

  Applying the estimates \eqref{eq:E1Def} and \eqref{eq:E2def}, the
  inequality \eqref{MainIneq2} becomes
  \begin{align*}
    & \intPi
    \abs{u_\Dx-u}\rho_{\alpha}(t-\tau)\omega_r(x-y)\rho_{r_0}(t-s)\dX
    \\ & \qquad \le \intPi
    \abs{u_\Dx-u}\rho_{\alpha}(t-\nu)\omega_r(x-y)\rho_{r_0}(t-s) \dX
    + E_1 + E_2.
  \end{align*}
  Note that both $E_1$ and $E_2$ are independent of $\alpha$. Thus, we
  can send $\alpha$ to zero, arriving at
$$
\kappa(\tau)\le \kappa(\nu) + E_1 + E_2,
$$
where $\kappa$ is defined as in Lemma~\ref{lem:kappa}.
By Lemma~\ref{lem:kappa} it follows that
\begin{align*}
  & \norm{u_\Dx(\cdot,\tau)-u(\cdot,\tau)}_{L^1(\R)} \\
  & \qquad \le \norm{u_\Dx(\cdot,\nu)-u(\cdot,\nu)}_{L^1(\R)} +
  2\left(L_c r_0 + \abs{u^0}_{BV(\R)}r\right) + E_1 + E_2.
\end{align*}
Recall that we had to pick $\nu > r_0$.  Denote by $L_d$ the $L^1$
Lipschitz constant of $t \mapsto u_\Dx(\cdot,t)$. By the triangle
inequality
\begin{align*}
  &\norm{u_\Dx(\cdot,\nu)-u(\cdot,\nu)}_{L^1(\R)} \\ & \qquad \le
  \norm{u_\Dx(\cdot,\nu)-u_\Dx^0}_{L^1(\R)} +
  \norm{u_\Dx^0-u^0}_{L^1(\R)} + \norm{u^0-u(\cdot,\nu)}_{L^1(\R)} \\
  &\qquad \le L_d\nu + \norm{u_\Dx^0-u^0}_{L^1(\R)} + L_c\nu.
\end{align*}
This means that
\begin{align*}
  \norm{u_\Dx(\cdot,\tau)-u(\cdot,\tau)}_{L^1(\R)} & \le \norm{u^0_\Dx
    - u^0}_{L^1(\R)} + \left(L_c+L_d\right)\nu \\ & \qquad +
  2\left(L_c r_0 + \abs{u^0}_{BV(\R)}r\right) + E_1 + E_2.
\end{align*}
Choose $r^3 = r_0^2 = \Dx$ and $\nu = 2r_0$. 
Then there exists a constant $C$ such that
\begin{equation*}\label{eq:MainResult}
  \norm{u_\Dx(\cdot,\tau)-u(\cdot,\tau)}_{L^1(\R)} 
  \le \norm{u^0_\Dx-u^0}+C \Dx^{\frac13}.
\end{equation*}

Now recall that $A(\sigma) = \hat{A}(\sigma) + \eta \sigma$ and so we 
need to send $\eta$ to zero to finish the proof. If $u_\eta$ is the
classical solution of the regularized equation and $u$ is the entropy
solution of the non-regularized equation, then it is well known that
$u_\eta(\cdot,t) \rightarrow u(\cdot,t)$ in $L^1(\R)$ as $\eta
\rightarrow 0$ (see section \ref{sec:prelim}).  Concerning the scheme
one may prove continuous dependence in $\ell^1$ on $\eta$ using
Gronwall's inequality.  Hence, we can also send $\eta$ to zero in the
scheme.  This finishes the proof of Theorem~\ref{MainResultSemiDisc}.
\end{proof}

\section{Implicit difference schemes}\label{sec: error estimate
  implicit}

In this section we show that the arguments presented in the previous
sections carry through for implicit schemes. Fix a time step $\Dt >
0$.  We consider implicit difference schemes of the form
\begin{equation}\label{DiscEqu}
  \Dmt u^n_j+ \Dm F(u_j^{n},u_{j+1}^{n}) 
  = \Dm\Dp A(u_j^{n}) \qquad n \ge 1,\ j \in \Z,
\end{equation}
where
\begin{equation*}
  \Dmt u^n_j = \frac{u^n_j-u^{n-1}_j}{\Dt}.
\end{equation*}
Let $t_n = n\Dt$ and $x_j = j\Dx$. We define the grid cells
\begin{equation*}
  I_j^n = [x_{j-1/2},x_{j+1/2}) \times (t_{n-1},t_{n}] \quad
  \text{for  $n \ge 0$ and $j \in \Z$.} 
\end{equation*}
The piecewise constant approximation is defined for all $(x,t) \in \R
\times (-\Dt,T]$ by
\begin{equation}\label{eq:uDeltaDef}
  u_\Delta(x,t) = u_j^n \, \mbox{ for } (x,t) \in I_j^n.
\end{equation}
The domain is chosen so that $\Dmt u_\Delta$ is defined for all $(x,t)
\in \R \times (0,T)$. For the existence of a unique solution $u^n_j$
to the nonlinear equation \eqref{DiscEqu} and the convergence of
$u_\Delta$ to an entropy solution, see \cite{Evje:1999et}.

We now state the main theorem.

\begin{theorem}\label{thm:discreterateImpl}
  Let $u$ be the entropy solution to \eqref{MainProblem}, and let
  $u_\Delta$ be defined via $u^n_j$ by \eqref{eq:uDeltaDef}, where
  $u^n_j$ solves \eqref{DiscEqu}.  If $u^0$ satisfies the same
  assumptions as in Theorem~\ref{MainResultSemiDisc}, then for all
  sufficiently small $\Dx$ and $\Dt$, and for all $n \in \N$ such that
  $t_n\in [0,T]$,
  \begin{equation*}%\label{eq:MainResultDisc}
    \norm{u_\Delta(\cdot,t_n)-u(\cdot,t_n)}_{L^1(\R)} 
    \le \norm{u^0_\Delta - u^0}_{L^1(\R)} + C\left(\Dx^{1/3}
      +\Dt^{1/2}\right),
  \end{equation*}
  where the constant $C_T$ depends on $u_0,A,f,T$, but not on
  $\Dx,\Dt$.
\end{theorem}

To prove this theorem we will follow step-by-step the proof of
Theorem~\ref{MainResultSemiDisc} and present the details whenever
there is a significant difference between the two cases.

Thanks to \cite[Lemma~2.4]{Evje:1999et}, we have the following $L^1$
Lipschitz continuity result:

\begin{lemma}\label{lem:L1TimeLipDisc}
  Let $m$ and $n$ be two non-negative integers. Then
  \begin{equation*}
    \norm{u_\Delta(\cdot,t_n)-u_\Delta(\cdot,t_m)}_{L^1(\R)} 
    \le L_d \abs{t_n-t_m},
  \end{equation*}
  where $L_d = \abs{F(u_j^0,u_{j+1}^0)-\Dp A(u_j^0)}_{BV}$.
\end{lemma}

Next, let us prove an implicit version of Lemma~\ref{lemma:
  EntropyCalcDiscrete}.
\begin{lemma}\label{lemma: EntropyCalcDiscreteImpl}
  Let $u_j^n$ be the solution to \eqref{DiscEqu}.  Then for all $c \in
  \R$,
  \begin{align*}%\label{SemiDiscEqu1}
    & \Dmt\psie(u_j^n,c) + \Dm Q^c(u_j^n,u_{j+1}^n) -\Dm\Dp
    \abseps{A(u_j^n)-A(c)} \\ & \qquad \le
    -\frac{1}{(\Dx)^2}\int^{u_j^n}_{u_{j+1}^n}
    \psie''(z,c)(A(z)-A(u_{j+1}^n))\dz \\ & \quad \qquad
    -\frac{1}{(\Dx)^2}\int^{u_j^n}_{u_{j-1}^n}
    \psie''(z,c)(A(z)-A(u_{j-1}^n))\dz,
  \end{align*}
  where $Q^c(u,v)$ is defined in Lemma~\ref{lemma:
    EntropyCalcDiscrete}.
\end{lemma}

\begin{proof}
  From \eqref{DiscEqu} it follows that
  \begin{equation*}
    \psie'(u_j^n,c)\Dmt u^n_j+ \psie'(u_j^n,c)\Dm F(u_j^{n},u_{j+1}^{n}) 
    = \psie'(u_j^n,c)\Dm\Dp A(u_j^{n}).
  \end{equation*}
  Apply Lemma~\ref{lemma:ChainRule2} with $g(\sigma) = \sigma$,
  $a=u_j^n$, and $b=u_j^{n-1}$ to obtain
  \begin{align*}
    \psie'(u_j^n,c)\Dmt u^n_j &= \Dmt \psie(u_j^n,c) - \frac{1}{\Dt}
    \int^{u_j^{n-1}}_{u_j^n} \psie''(z,c)(z-u_j^{n-1})\dz \\ & \ge
    \Dmt \psie(u_j^n,c).
  \end{align*}
  The remaining part of the proof follows exactly as in the proof of
  Lemma~\ref{lemma: EntropyCalcDiscrete}.
\end{proof}

Let us define the time shift operator
$$
S^t_\Dt \sigma(t) = \sigma(t + \Dt),
$$
for any function $\sigma = \sigma(t)$.

\begin{lemma}\label{lemma:MainLemmaTestFuncImpl}
  Suppose $A' > 0$. Let $u_\Delta = u_\Delta(x,t)$ be defined by
  \eqref{eq:uDeltaDef}, and let $u=u(y,s)$ be the classical solution
  of \eqref{MainProblem}.  Let $\psi(t) := \car{[\nu,\tau)}(t)$ and
  define
$$
\test(x,t,y,s) = \psi(t)\omega_r(x-y)\rho_{r_0}(t-s),
$$
where $\omega_r,\rho_{r_0},\nu,\tau$ are chosen as in
Lemma~\ref{lemma:MainLemmaTestFunc}.  Then
\begin{align*}\label{MainIneq2Impl}
  & \intPi \abs{u_\Delta-u}\delta_\Dt^-(t-\nu)\omega_r\rho_{r_0} \dX  \\
  & \qquad + \intPi \abs{u_\Delta-u}
  S^t_\Dt\psi \omega_r (\Dpt \rho_{r_0} - \partial_t\rho_{r_0})\dX \\
  &\qquad + \Dt \intPi \abs{u_\Delta-u}\Dpt \psi \omega_r \partial_s \rho_{r_0} \dX \\
  &\qquad +\intPi \sgn{u_\Delta-u}\left(f(u_\Delta)-f(u)\right)
  \left(\Dp\test +\test_{y}\right)\dX \\
  &\qquad +\intPi \left(\int_{u_\Dx}^{S_\Dx u_{\Delta}}
    \sgn{z-u}F_2'(z)\,dz \right)\,\Dp\test \dX \\
  &\qquad +\intPi \abs{A(u_\Delta)-A(u)}
  \left(\Dm\Dp\test + (\Dp + \Dm)\test_{y} + \test_{yy}\right) \dX \\
  &\qquad \qquad \ge \intPi
  \abs{u_\Delta-u}\delta_\Dt^-(t-\tau)\omega_r\rho_{r_0} \dX +
  \liminf_{\varepsilon \downto 0} \intPi \Ee_\Delta \test \dX,
\end{align*}
where
$$
\delta_\Dt^-(t) = \frac{1}{\Dt}\car{[-\Dt,0)}(t),
$$
and $\Ee_\Delta(x,t,y,s) = \Ee[u](u_{j-1}^n,u_j^n,u_{j+1}^n)(y,s)$ for
$(x,t) \in I_j^n$.
\end{lemma}

\begin{proof}
  As in Lemma~\ref{lemma:DoubelingLemma}, we obtain by
  Lemma~\ref{lemma: EntropyCalcDiscreteImpl} the inequality
  \begin{align*}
    & \Dmt\psie(u_j^n,u) + \partial_s\psie(u,u_j^n) + \partial_y
    \qe(u,u_j^n) + \Dm Q^u(u_j^n,u_{j+1}^n) \\ & \qquad\quad
    -( \partial_y^2 +\partial_y(\Dm +\Dp) +
    \Dm\Dp)\abseps{A(u_j^n)-A(u)}) \le -\Ee_{j,n},
  \end{align*}
  where $\Ee_{j,n} := \Ee[u](u_{j-1}^n,u_j^n,u_{j+1}^n)$ is defined in
  Lemma~\ref{lemma:DoubelingLemma}. Let us multiply by $\test$ and
  integrate over $\Pi_T^2$. Integration by parts for difference
  quotients and ordinary integration by parts gives
  \begin{align*}
    &\intPi \psie(u_\Delta,u)\Dpt\test + \psie(u,u_\Delta)\test_s \dX
    \\ & \qquad +\intPi \qe(u,u_\Delta)\test_y + Q^u(u_\Delta,S_\Dx
    u_\Delta)\Dp\test \dX \\ & \qquad +\intPi
    \abseps{A(u_\Delta)-A(u)}(\test_{yy} +(\Dm +\Dp)\test_y +
    \Dm\Dp\test) \dX \\ & \qquad\qquad \ge \intPi \Ee_\Delta \test
    \dX.
  \end{align*}
  Consider the first term on the left. Let $\varepsilon$ tend to zero
  as in the proof of Lemma~\ref{lemma:MainLemmaTestFunc}. Using the
  Leibniz rule for difference quotients and adding and subtracting we
  obtain
$$
\Dpt \test = S^t_\Dt \psi \omega_r \Dpt \rho_{r_0} + \Dpt \psi
\omega_r \rho_{r_0}.
$$
Furthermore,
$$
\test_s = -S^t_\Dt\psi \omega_r \partial_t\rho_{r_0} + \Dt\Dpt \psi
\omega_r \partial_s\rho_{r_0}.
$$
Hence,
\begin{align*}%\label{eq:TimeDerSplit}
  & \intPi \abs{u_\Delta-u} \left(\Dpt \test + \test_s\right)\dX \\ &
  \quad = \intPi \abs{u_\Delta-u}S^t_\Dt \psi \omega_r (\Dpt
  \rho_{r_0} - \partial_t\rho_{r_0})\dX \\ & \quad \qquad+ \Dt \intPi
  \abs{u_\Delta-u} \Dpt \psi \omega_r \partial_s \rho_{r_0} \dX
  % \\ & \qquad \qquad
  + \intPi \abs{u_\Delta-u} \Dpt \psi \omega_r \rho_{r_0} \dX.
\end{align*}
Finally, we use that
\begin{equation}\label{eq:psiDiff}
  \Dpt \psi = \delta_\Dt^-(t-\nu) -  \delta_\Dt^-(t-\tau).
\end{equation}
The lemma now follows, as in the proof of
Lemma~\ref{lemma:MainLemmaTestFunc}, by letting $\varepsilon$ tend to
zero.
\end{proof}

Comparing the terms in Lemma~\ref{lemma:MainLemmaTestFunc} with the terms in Lemma~\ref{lemma:MainLemmaTestFuncImpl} we
recognize all but two terms.
\begin{estimate}\label{est:TimeDiff}
$$
\left|\intPi \abs{u-u_\Delta}S^t_\Dt \psi \omega_r (\Dpt \rho_{r_0}
  - \partial_t\rho_{r_0})\dX\right| \le C\frac{\Dt}{r_0}\left(1 +
  \frac{\Dt}{r_0}\right).
$$
\end{estimate}

\begin{proof}
  To show this we use a Taylor expansion:
  \begin{align*}
    & \rho_{r_0}(t + \Dt-s)- \rho_{r_0}(t-s) \\ & \quad
    = \int_t^{t + \Dt} \frac{\partial}{\partial z} \rho_{r_0}(z-s) \dz \\
    & \quad = \frac{\partial}{\partial z} \bigg|_{z = t}
    \rho_{r_0}(z-s)\Dt - \int_t^{t+\Dt} \frac{\partial^2}{\partial
      z^2} \rho_{r_0}(z-s)(z-(t+\Dt))\dz.
  \end{align*}
  It follows that
  \begin{equation}\label{eq:TaylorRhoR0}
    \Dpt \rho_{r_0} - \partial_t\rho_{r_0} 
    =-\frac{1}{\Dt}\int_t^{t+\Dt} \frac{\partial^2}{\partial z^2}
    \rho_{r_0}(z-s)(z-(t+\Dt))\dz. 
  \end{equation}
  Integration by parts yields
  \begin{align*}
    -\frac{1}{\Dt}&\int_t^{t+\Dt}\int_0^T
    \abseps{u-u_\Delta}\frac{\partial^2}{\partial z^2}
    \rho_{r_0}(z-s)(z-(t+\Dt))\,dsdz \\
    &= \frac{1}{\Dt}\int_t^{t+\Dt}\int_0^T
    \abseps{u-u_\Delta}\frac{\partial}{\partial s}
    \frac{\partial}{\partial z} \rho_{r_0}(z-s)(z-(t+\Dt))\,dsdz \\ &=
    -\frac{1}{\Dt}\int_t^{t+\Dt}\int_0^T
    \signe(u-u_\Delta)u_s\frac{\partial}{\partial z}
    \rho_{r_0}(z-s)(z-(t+\Dt))\,dsdz.
  \end{align*}
  Since
$$
\frac{\partial}{\partial z} \rho_{r_0}(z-s) =
\frac{1}{r_0}\frac{\partial}{\partial z}
\rho\left(\frac{z-s}{r_0}\right) = \frac{1}{r_0^2}
\rho'\left(\frac{z-s}{r_0}\right)
$$
and $\rho_{r_0}$ has support in $[-r_0,r_0]$, it follows that
\begin{align*}
  & \frac{1}{\Dt}\abs{\int_t^{t+\Dt}\int_0^T
    \abseps{u(y,s)-u_\Delta(x,t)}\frac{\partial^2}{\partial z^2}
    \rho_{r_0}(z-s)(z-(t+\Dt))\,dsdz} \\
  & \qquad \le \frac{C}{r_0^2\Dt} \int_0^T
  |u_s(y,s)|\int_t^{t+\Dt}\car{|z-s|
    \le r_0}|z-(t+\Dt)|\,dzds \\
  & \qquad \le C\frac{\Dt}{r_0^2} \int_0^T |u_s|\car{|t-s| \le r_0 +
    \Dt}\,ds.
\end{align*}
Multiply the above inequality by $\psi(t)\omega_r(x-y)$ and integrate
in $x,y,t$. From the resulting inequality and \eqref{eq:TaylorRhoR0},
we arrive at the estimate
\begin{align*}
  & \abs{\intPi \abseps{u-u_\Delta} S^t_\Dt \psi \omega_r
    (\Dpt \rho_{r_0} - \partial_t\rho_{r_0})\dX} \\
  & \qquad \le C\frac{\Dt}{r_0^2} \intPi|u_s|S^t_\Dt\psi \omega_r
  \car{|t-s| \le r_0 + \Dt}\dX \le C\frac{\Dt}{r_0}\left(1 +
    \frac{\Dt}{r_0}\right)\|u_s\|_{L^1(\Pi_T)}.
\end{align*}
Since $\|u_s(\cdot,s)\|_{L^1(\R)}$ is uniformly bounded on $[0,T]$,
the estimate follows from the dominated convergence theorem.
\end{proof}

\begin{estimate}\label{est:TimeDiffTwo}
$$
\left|\Dt \intPi \abs{u-u_\Delta}\Dpt \psi \omega_r
  \partial_s \rho_{r_0} \dX \right| \le C\frac{\Dt}{r_0}.
$$
\end{estimate}

\begin{proof}
  Integration by parts yields
$$
\left|\intPi \abs{u-u_\Delta}\Dpt \psi \omega_r \partial_s \rho_{r_0}
  \dX\right| \le \intPi \abs{u_s} \abs{\Dpt \psi} \omega_r \rho_{r_0}
\dX.
$$
Because of \eqref{eq:psiDiff} and since
$$
\norm{\rho_{r_0}}_{L^\infty} \le \frac{\norm{\rho}_{L^\infty}}{r_0},
$$
it follows that
$$
\intPi \abs{u_s} \abs{\Dpt \psi} \omega_r \rho_{r_0} \dX \le
2\frac{\norm{\rho}_{L^\infty}}{r_0} \norm{u_s}_{L^1(\Pi_T)}.
$$
\end{proof}

\begin{proof}[Proof of Theorem~\ref{thm:discreterateImpl}]
  We start out from Lemma~\ref{lemma:MainLemmaTestFuncImpl} with
  $A(\sigma) = \hat{A}(\sigma) + \eta\sigma$, where $\hat{A}$ is the
  original degenerate diffusion function.  By
  Estimate~\ref{est:TimeDiff} and Estimate \ref{est:TimeDiffTwo},
  \begin{multline}\label{eq:E3Def}
    \intPi \abs{u-u_\Delta}S^t_\Dt\psi \omega_r (\Dpt \rho_{r_0}
    - \partial_t\rho_{r_0})\dX \\ +\Dt \intPi \abs{u-u_\Delta}\Dpt
    \psi \omega_r \partial_s \rho_{r_0} \dX \le C\frac{\Dt}{r_0}
    \left(1 + \frac{\Dt}{r_0}\right) =: E_3.
  \end{multline}
  Since all the estimates from Section~\ref{sec:Estimates} apply, we
  obtain
  \begin{align*}
    &\intPi \abs{u_\Delta-u}\delta^\Dt(t-\tau)
    \omega_r(x-y)\rho_{r_0}(t-s)\dX \\ & \quad \le \intPi
    \abs{u_\Dx-u}\delta^\Dt(t-\nu) \omega_r(x-y)\rho_{r_0}(t-s) \dX +
    E_1 + E_2 + E_3,
  \end{align*}
  where $E_1$ and $E_2$ are defined respectively in \eqref{eq:E1Def}
  and \eqref{eq:E2def}.

Let us make the simplifying assumption that $\nu = t_m$ 
and $\tau = t_n$ for some $m,n \in \N$.  
Then the above inequality rewrites as
$$
\kappa(t_n) \le \kappa(t_m)+E_1 + E_2 + E_3,
$$
where
$$
\kappa(t) = \int_\R \int_{\Pi_T}
\abs{u_\Delta(x,t)-u(y,s)}\omega_r(x-y) \rho_{r_0}(t-s) \,dydsdx.
$$
Applying Lemmas \ref{lem:kappa} and \ref{lem:L1TimeLipDisc}, and
following the reasoning given in the semi-discrete case, we arrive at
\begin{align*}%\label{eq:MainIneqDisc}
  &\norm{u_\Delta(\cdot,t_n)-u(\cdot,t_n)}_{L^1(\R)} \\ & \quad \qquad
  \le \norm{u^0_\Delta - u^0}_{L^1(\R)} \\ & \quad \qquad \qquad \quad
  + \left(L_c+L_d\right)t_m +2\left(L_c r_0 +
    \abs{u^0}_{BV(\R)}r\right) \\ & \quad \qquad\qquad \quad
  +C(1+r+\Dx)^2 \left(1 + \frac{\Dx}{r}\right)^3\frac{\Dx}{r^2} +
  C\frac{\Dx}{r_0} + C\frac{\Dt}{r_0}\left(1 + \frac{\Dt}{r_0}\right)
  \\ & \quad \qquad \le \norm{u^0_\Delta - u^0}_{L^1(\R)} +
  C\left(\frac{\Dx}{r^2} + \frac{\Dx + \Dt}{r_0} + r + r_0\right),
\end{align*}
where $L_d$ is the constant in Lemma~\ref{lem:L1TimeLipDisc} and $L_c$
is the constant from Lemma~\ref{lem:kappa}. Minimizing over $r$ and
$r_0$, it is straightforward to see that for sufficiently small $\Dt$,
the minimum of the last term is dominated by
$$
C\left(\Dx^{1/3}+\Dt^{1/2}\right).
$$
This proves the theorem. 
\end{proof}

\section{Explicit difference schemes}\label{sec: error estimate explicit}
In this section we use the techniques developed in the previous
section to provide a similar result concerning the explicit
scheme. Fix a time step $\Dt > 0$. We consider explicit schemes of the
form
\begin{equation}\label{DiscEquExpl}
  \Dpt u^n_j+ \Dm F(u_j^{n},u_{j+1}^{n}) 
  = \Dm\Dp A(u_j^{n}) \qquad n \ge 1,\ j \in \Z,
\end{equation}
where
$$
\Dpt u_j^n = \frac{u_j^{n+1}-u_j^n}{\Dt}.
$$
The relevant a priori estimates and convergence to an entropy solution
is proved in \cite{Evje:2000vn} under the hypothesis
\begin{equation}\label{eq:CFL}
  1 - \frac{\Dt}{\Dx}(F_1'(z) - F_2'(z)) - 2\frac{\Dt}{\Dx^2}A'(w) 
  \ge 0, \quad \forall (z,w) \in \R^2. 
\end{equation}
Let $t_n = n\Dt$ and $x_j = j\Dx$. We define the grid cells
$$
I_j^n = [x_{j-1/2},x_{j+1/2}) \times [t_{n},t_{n+1}), \quad \text{for
  $n \ge 0$ and $j \in \Z$.}
$$
The piecewise constant approximation is defined for all $(x,t) \in
\Pi_T$ by
\begin{equation}\label{eq:uDeltaExplDef}
  u_\Delta(x,t) = u_j^n \, \mbox{ for } (x,t) \in I_j^n.
\end{equation}

\begin{theorem}\label{thm:discreterateExpl}
  Let $u$ be the entropy solution to \eqref{MainProblem}, and let
  $u_\Delta$ be defined by \eqref{eq:uDeltaExplDef} via $u^n_j$, where
  $u^n_j$ solves \eqref{DiscEquExpl}.
  Suppose $\Dt$ and $\Dx$ are chosen such that \eqref{eq:CFL} and the
  strengthened condition $\Dt \le C\Dx^{8/3}$ hold.  If $u^0$
  satisfies the same assumptions as in
  Theorem~\ref{MainResultSemiDisc}, then for all sufficiently small
  $\Dx$, and for all $n \in \N$ such that $t_n\in [0,T]$,
  \begin{equation*}%\label{eq:MainResultDisc}
    \norm{u_\Delta(\cdot,t_n)-u(\cdot,t_n)}_{L^1(\R)} 
    \le \norm{u^0_\Delta - u^0}_{L^1(\R)} + C\Dx^{1/3},
  \end{equation*}
  where the constant $C_T$ depends on $A,f,u^0, T$, but not on $\Dx$.
\end{theorem}

We begin by proving the following lemma
\begin{lemma}\label{lemma:EntropyCalcDiscExpl}
  Let $\seq{u_j^n}$ be the solution to \eqref{DiscEquExpl}.  Suppose
  that
  \begin{equation}\label{CFLConditionCons}
    1-\frac{\Dt}{\Dx}(F_1'(z)-F_2'(z)) \ge 0,
  \end{equation}
  for all $z \in \R$. Then for all $j \in \Z$ and $n \in \N$,
  \begin{align*}
    & \Dpt \psie(u_j^n,c)
    + \Dm Q^c(u_j^n,u_{j+1}^n)  - \Dm\Dp\abseps{A(u_j^n)-A(c)} \\
    & \qquad \le -\frac{1}{(\Dx)^2}\int^{u_j^n}_{u_{j-1}^n} \psie''(z,c)(A(z)-A(u_{j-1}^n))\dz \\
    & \qquad\quad
    - \frac{1}{(\Dx)^2}\int^{u_j^n}_{u_{j+1}^n} \psie''(z,c)(A(z)-A(u_{j+1}^n))\dz \\
    &\qquad\quad + \left(\int_{u_j^n}^{u_j^{n+1}} \psie''(z,c)\dz
    \right)\Dm\Dp A(u_j^n),
  \end{align*}
  where $Q^c(u,v)$ is defined by
$$
Q^{c}(u,v) = \int_c^u\psie'(z,c)F_1'(z) \dz + \int_c^v
\psie'(z,c)F_2'(z) \dz, \qquad u,v\in \R.
$$
\end{lemma}

\begin{proof}
  We divide the proof into two steps. \\*
  \emph{Claim 1.} Let $\seq{u_j^n}$ be a solution to
  \eqref{DiscEquExpl}. Then
  \begin{align*}
    &\Dpt \psie(u_j^n,c)  + \Dm Q^c(u_j^n,u_{j+1}^n) - \Dm\Dp\abseps{A(u_j^n)-A(c)} \\
    & \qquad =-E^c(u_j^n,u_j^{n+1},u_{j+1}^n,u_{j-1}^n),
  \end{align*}
  where
  \begin{equation*}%\label{eq:DissipTermDisc}
    \begin{split}
      &E^c(u_j^n,u_j^{n+1},u_{j+1}^n,u_{j-1}^n) \\ & \quad =
      \frac{1}{\Dt}\int_{u_j^n}^{u_j^{n+1}}
      \psie''(z,c)(z-u_j^{n+1})\dz \\ & \qquad
      +\frac{1}{\Dx}\int^{u_j^n}_{u_{j-1}^n}
      \psie''(z,c)\left[(F_1(z)-F_1(u_{j-1}^n))+\frac{1}{\Dx}(A(z)-A(u_{j-1}^n))\right]\dz
      \\ & \qquad + \frac{1}{\Dx}\int_{u_j^n}^{u_{j+1}^n}
      \psie''(z,c)\left[(F_2(z)-F_2(u_{j+1}^n))-\frac{1}{\Dx}(A(z)-A(u_{j+1}^n))\right]\dz.
    \end{split}
  \end{equation*}
  \emph{Proof of Claim 1.} By definition \eqref{DiscEquExpl} of
  $\seq{u_j^n}$ it follows that
$$
\psie'(u_j^n,c)\left[ \Dpt u_j^n + \Dm F(u_j^n,u_{j+1}^n) - \Dm\Dp
  A(u_j^n)\right] = 0.
$$
Let $g(z) = z$ in Lemma~\ref{lemma:ChainRule2}. It follows that
$$
\psie'(u_j^n,c)\Dpt u_j^n = \Dpt \psie(u_j^n,c) + \frac{1}{\Dt}
\int_{u_j^n}^{u_j^{n+1}} \psie''(z,c)(z-u_j^{n+1})\,dz.
$$
The remaining terms can be treated as in
Lemma~\ref{lemma: EntropyCalcDiscrete}. \\*
\emph{Claim 2.} Suppose \eqref{CFLConditionCons} holds. Then
\begin{equation}\label{eq:Claim2ExplEntropy}
  \begin{split}
    &\frac{1}{\Dt}\int_{u_j^n}^{u_j^{n+1}} \psie''(z,c)(z-u_j^{n+1})\dz  \\
    &\qquad +\frac{1}{\Dx}\int^{u_j^n}_{u_{j-1}^n}
    \psie''(z,c)(F_1(z)-F_1(u_{j-1}^n))\dz \\ &\qquad
    +\frac{1}{\Dx}\int_{u_j^n}^{u_{j+1}^n}
    \psie''(z,c)(F_2(z)-F_2(u_{j+1}^n))\dz \\ & \qquad \quad \qquad
    \ge -\int_{u_j^n}^{u_j^{n+1}}\psie''(z,c)\Dm\Dp A(u_j^n) \dz.
  \end{split}
\end{equation}
\emph{Proof of Claim 2.} Consider the first term on the left-hand side
of \eqref{eq:Claim2ExplEntropy}. By definition, $u_j^{n+1} = u_j^n -
\Dt\Dm F(u_j^n,u_{j+1}^n) + \Dt \Dm\Dp A(u_j^n)$, and so
\begin{align*}
  &\frac{1}{\Dt}\int_{u_j^n}^{u_j^{n+1}} \psie''(z,c)(z-u_j^{n+1})\dz
  \\ & \quad = \frac{1}{\Dt}\int_{u_j^n}^{u_j^{n+1}}
  \psie''(z,c)(z-u_j^n)\dz
  % \\ & \qquad
  +\int_{u_j^n}^{u_j^{n+1}}\psie''(z,c)\Dm F(u_j^n,u_{j+1}^n)\dz \\ &
  \quad \qquad\qquad - \int_{u_j^n}^{u_j^{n+1}} \psie''(z,c)\Dm\Dp
  A(u_j^n)\dz =: T_1 + T_2 + T_3.
\end{align*}
Note that $T_1$ is positive. Let us split $T_2$ according to
$$
\Dm F(u_j^n,u_{j+1}^n) =
\frac{1}{\Dx}\left(F_1(u_j^n)-F_1(u_{j-1}^n)\right) +
\frac{1}{\Dx}\left(F_2(u_{j+1}^n)-F_2(u_j^n)\right),
$$
and thus
\begin{align*}
  T_2 &= \frac{1}{\Dx}\int_{u_j^n}^{u_j^{n+1}}
  \psie''(z,c)\left(F_1(u_j^n)-F_1(u_{j-1}^n)\right) \dz \\ & \qquad +
  \frac{1}{\Dx}\int_{u_j^n}^{u_j^{n+1}}
  \psie''(z,c)\left(F_2(u_{j+1}^n)-F_2(u_j^n)\right) \dz.
\end{align*}
Now, let us split the two other terms appearing in equation 
\eqref{eq:Claim2ExplEntropy}:
\begin{align*}
  S_1 & := \frac{1}{\Dx}\int^{u_j^n}_{u_{j-1}^n}
  \psie''(z,c)(F_1(z)-F_1(u_{j-1}^n))\dz \\ & =
  \frac{1}{\Dx}\int^{u_j^{n+1}}_{u_{j-1}^n}
  \psie''(z,c)(F_1(z)-F_1(u_{j-1}^n))\dz \\ & \qquad
  -\frac{1}{\Dx}\int_{u_j^n}^{u_j^{n+1}}
  \psie''(z,c)(F_1(z)-F_1(u_{j-1}^n))\dz
\end{align*}
and
\begin{align*}
  S_2 & := \frac{1}{\Dx}\int_{u_j^n}^{u_{j+1}^n}
  \psie''(z,c)(F_2(z)-F_2(u_{j+1}^n))\dz \\ & =
  -\frac{1}{\Dx}\int^{u_j^{n+1}}_{u_{j+1}^n}
  \psie''(z,c)(F_2(z)-F_2(u_{j+1}^n))\dz \\ & \qquad
  +\frac{1}{\Dx}\int^{u_j^{n+1}}_{u_j^n}
  \psie''(z,c)(F_2(z)-F_2(u_{j+1}^n))\dz.
\end{align*}
Combining the above expressions we obtain
\begin{align*}
  T_2 + S_1 + S_2 & =-\frac{1}{\Dx}\int_{u_j^n}^{u_j^{n+1}}
  \psie''(z,c)(F_1(z)-F_1(u_j^n))\dz \\ & \qquad
  +\frac{1}{\Dx}\int_{u_j^n}^{u_j^{n+1}}
  \psie''(z,c)(F_2(z)-F_2(u_j^n))\dz \\ & \qquad
  +\frac{1}{\Dx}\int^{u_j^{n+1}}_{u_{j-1}^n}
  \psie''(z,c)(F_1(z)-F_1(u_{j-1}^n))\dz \\ & \qquad
  -\frac{1}{\Dx}\int^{u_j^{n+1}}_{u_{j+1}^n}
  \psie''(z,c)(F_2(z)-F_2(u_{j+1}^n))\dz.
\end{align*}
The two last terms on the right-hand side are positive as $F$ is
monotone. Let
$$
H(z) = z - \frac{\Dt}{\Dx}\left(F_1(z)-F_2(z)\right).
$$
Then, by assumption \eqref{CFLConditionCons},
$$
T_1 + T_2 + S_1 + S_2 \ge \frac{1}{\Dt}\int_{u_j^n}^{u_j^{n+1}}
\psie''(z,c)\left[H(z)-H(u_j^n)\right]\dz \ge 0.
$$
Adding $T_3$ to both sides proves Claim 2.

By Claim 2
\begin{equation*}
  \begin{split}
    & E^c (u_j^n,u_j^{n+1},u_{j+1}^n,u_{j-1}^n) \\ & \quad \ge
    \frac{1}{\Dx^2}\int^{u_j^n}_{u_{j-1}^n}
    \psie''(z,c)(A(z)-A(u_{j-1}^n))\dz \\ & \quad \qquad
    +\frac{1}{\Dx^2}\int^{u_j^n}_{u_{j+1}^n}
    \psie''(z,c)(A(z)-A(u_{j+1}^n))\dz \\ & \quad \qquad
    -\int_{u_j^n}^{u_j^{n+1}} \psie''(z,c)\Dm\Dp A(u_j^n)\dz.
  \end{split}
\end{equation*}
Combining this with Claim 1 proves the lemma.
\end{proof}

\begin{lemma}\label{lemma:MainLemmaTestFuncExpl}
  Suppose $A' > 0$, and \eqref{CFLConditionCons} applies.  Let
  $u_\Delta = u_\Delta(x,t)$ be defined by \eqref{eq:uDeltaExplDef},
  and let $u=u(y,s)$ be the classical solution of \eqref{MainProblem}.
  Set $\psi(t) := \car{[\nu,\tau)}(t)$ and define
$$
\test(x,t,y,s) = \psi(t)\omega_r(x-y)\rho_{r_0}(t-s),
$$
where $\omega_r,\rho_{r_0},\nu,\tau$ are chosen as in
Lemma~\ref{lemma:MainLemmaTestFunc}. Then
\begin{align*}\label{MainIneq2Impl}
  \intPi  & \abs{u_\Delta-u}\delta_\Dt^+(t-\nu)\omega_r\rho_{r_0} \dX  \\
  &+ \intPi \abs{u_\Delta-u}S^t_{-\Dt}\psi \omega_r (\Dmt \rho_{r_0} - \partial_t\rho_{r_0})\dX \\
  &+ \Dt\intPi \abs{u_\Delta-u}\Dmt \psi \omega_r \partial_s\rho_{r_0} \dX \\
  &+\intPi \sgn{u_\Delta-u}\left(f(u_\Delta)-f(u)\right) \left(\Dp\test +\test_{y}\right)\dX \\
  &+\intPi \left(\int_{u_\Dx}^{S_\Dx u_{\Delta}} \sgn{z-u}F_2'(z)\,dz \right)\,\Dp\test \dX \\
  &+\intPi \abs{A(u_\Delta)-A(u)} \left(\Dm\Dp\test + (\Dp + \Dm)\test_{y} + \test_{yy}\right) \dX \\
  & \qquad \ge \intPi \abs{u_\Delta-u}\delta_\Dt^+(t-\tau)\omega_r\rho_{r_0} \dX  \\
  & \qquad \qquad
  + \liminf_{\varepsilon \downto 0} \intPi \Ee_\Delta \test \dX \\
  & \qquad\qquad - \Dt \intPi \Dpt \signe(A(u_\Delta)-A(u)) \Dm\Dp
  A(u_\Delta)\test\dX,
\end{align*}
where
$$
\delta_\Dt^+(t) = \frac{1}{\Dt}\car{[0,\Dt)}(t),
$$
and $\Ee_\Delta(x,t,y,s) = \Ee[u](u_{j-1}^n,u_j^n,u_{j+1}^n)(y,s)$ for
$(x,t) \in I_j^n$.
\end{lemma}

\begin{proof}
  By Lemma~\ref{lemma:EntropyCalcDiscExpl}, we obtain as in
  Lemma~\ref{lemma:DoubelingLemma} the following inequality:
  \begin{align*}
    & \Dpt\psie(u_j^n,u) + \partial_s\psie(u,u_j^n) + \partial_y
    \qe(u,u_j^n) + \Dm Q^u(u_j^n,u_{j+1}^n) \\ & \qquad \qquad
    -( \partial_y^2 +\partial_y(\Dm +\Dp) +
    \Dm\Dp)\abseps{A(u_j^n)-A(u)}) \\ & \qquad\qquad \qquad \le
    -\Ee_{j,n} + \left(\int_{u_j^n}^{u_j^{n+1}} \psie''(z,u)\dz
    \right)\Dm\Dp A(u_j^n),
  \end{align*}
  where $\Ee_{j,n} := \Ee[u](u_{j-1}^n,u_j^n,u_{j+1}^n)$ is defined in
  Lemma~\ref{lemma:DoubelingLemma}. Note that
$$
\int_{u_j^n}^{u_j^{n+1}} \psie''(z,u)\dz = \Dt \,\Dpt
\signe(A(u_j^n)-A(u)).
$$
Integration by parts for difference quotients and ordinary integration
by parts gives
\begin{align*}
  &\intPi \psie(u_\Delta,u)\Dmt\test + \psie(u,u_\Delta)\test_s \dX \\
  & \quad +\intPi \qe(u,u_\Delta)\test_y + Q^u(u_\Delta,S_\Dx
  u_\Delta)\Dp\test \dX \\ & \quad +\intPi
  \abseps{A(u_\Delta)-A(u)}(\test_{yy} +(\Dm +\Dp)\test_y +
  \Dm\Dp\test) \dX \\ & \quad\quad \ge \intPi \Ee_\Delta \test \dX -
  \Dt \intPi \Dpt \signe(A(u_\Dt)-A(u)) \Dm\Dp A(u_\Delta)\test\dX.
\end{align*}
Consider the first term on the left-hand side. Let $\varepsilon$ tend
to zero as in the proof of Lemma~\ref{lemma:MainLemmaTestFunc}. Using
the Leibniz rule for difference quotients we obtain
$$
\Dmt \test = S^t_{-\Dt}\psi \omega_r \Dmt \rho_{r_0} + \Dmt \psi
\omega_r \rho_{r_0}.
$$
Recall that $\partial_s\rho_{r_0} = -\partial_t\rho_{r_0}$, so adding
and subtracting gives
$$
\test_s = -S^t_{-\Dt}\psi \omega_r \partial_t \rho_{r_0} + \Dt \Dmt
\psi \omega_r \partial_s \rho_{r_0}.
$$
Hence,
\begin{align*}%\label{eq:TimeDerSplit}
  & \intPi \abs{u_\Delta-u}\left(\Dmt \test + \test_s\right)\dX \\ &
  \qquad = \intPi \abs{u_\Delta-u} S^t_{-\Dt}\psi \omega_r (\Dmt
  \rho_{r_0} - \partial_t\rho_{r_0})\dX \\ & \qquad \qquad + \Dt\intPi
  \abs{u_\Delta-u} \Dmt \psi \omega_r \partial_s\rho_{r_0} \dX
  % \\ & \qquad \qquad
  + \intPi \abs{u_\Delta-u}\Dmt \psi \omega_r \rho_{r_0} \dX.
\end{align*}
Finally, we use that
\begin{equation}\label{eq:psiTimeDiff}
  \Dmt \psi  = \delta_\Dt^+(t-\nu) -  \delta_\Dt^+(t-\tau).
\end{equation}
Concerning the second term on the left-hand side, we apply
\eqref{eq:NumFluxPres}. The lemma now follows by sending $\varepsilon$
to zero, as in the proof of Lemma~\ref{lemma:MainLemmaTestFunc}.
\end{proof}

As seen by comparing Lemmas \ref{lemma:MainLemmaTestFuncExpl} and 
\ref{lemma:MainLemmaTestFuncImpl}, there is one new term. To
estimate this term we will use a result from
\cite[p.~1853]{Evje:2000vn}.

\begin{lemma}\label{lemma:HolderContDiff}
  Let $u_j^n$ be the solution to \eqref{DiscEquExpl}.  Suppose the CFL
  condition \eqref{eq:CFL} is satisfied.  Then there exists a constant
  $L$ such that
$$
\Dx \sum_j \abs{\Dp A(u_j^m)-\Dp A(u_j^n)} \le L\sqrt{(m-n)\Dt},
\qquad \text{for all $m \ge n$.}
$$
\end{lemma}

\begin{estimate}\label{est:TimeErrExpl} 
  Suppose \eqref{eq:CFL} is satisfied. Then
  \begin{align*}
    & \left|\Dt \intPi \Dpt \signe(A(u_\Delta)-A(u)) \Dm\Dp
      A(u_\Delta)\test\dX \right| \\ & \qquad\qquad \le
    C\frac{\sqrt{\Dt}}{\Dx} + C \frac{\Dt}{r_0}\left(1 +
      \frac{\Dt}{r_0}\right) + C \Dt.
  \end{align*}
\end{estimate}

\begin{proof}
  Integration by parts for difference quotients gives
  \begin{align*}
    & \Dt\intPi \Dpt \signe(A(u_\Delta)-A(u)) \Dm\Dp
    A(u_\Delta)\test\dX \\ &\quad = -\Dt\intPi
    \signe(A(u_\Delta)-A(u)) S^t_{-\Dt}\Dm\Dp A(u_\Delta)\Dmt\test \dX
    \\ & \quad \qquad -\Dt\intPi \signe(A(u_\Delta)-A(u))\Dmt\Dm\Dp
    A(u_\Delta)\test\dX =:T_1 + T_2,
  \end{align*}
  where we have used that
\begin{equation*}
  \Dmt(\Dm\Dp A(u_\Delta)\test) = S^t_{-\Dt}\Dm\Dp A(u_\Delta)\Dmt\test
  + \Dmt\Dm\Dp A(u_\Delta)\test.
\end{equation*}
Let us consider $T_1$ first. By the Leibniz rule for 
difference quotients,
\begin{align*}
  T_1 &= \Dt\intPi \signe(A(u_\Delta)-A(u)) S^t_{-\Dt}\Dm\Dp
  A(u_\Delta)S^t_{-\Dt}\psi \omega_r \Dmt \rho_{r_0} \dX \\ &\qquad +
  \Dt\intPi \signe(A(u_\Delta)-A(u)) S^t_{-\Dt}\Dm\Dp
  A(u_\Delta)\Dmt\psi\omega_r\rho_{r_0}\dX \\ & =:T_{1,1} +T_{1,2}.
\end{align*}

Using equation \eqref{eq:psiTimeDiff},
$$
\abs{T_{1,2}} \le \Dt\int_{\Pi_T} \abs{S^t_{-\Dt}\Dm\Dp A(u_\Delta)}
\left(\abs{\delta_\Dt^+(t-\nu)} + \abs{\delta_\Dt^+(t-\tau)}\right)
\dxdt \le C\Dt,
$$
as $\norm{\Dm\Dp A(u_\Delta(\cdot,t))}_{L^1(\R)}$ is bounded
independent of $\Delta$ and $t$ (\cite[Lemma~3.4]{Evje:2000ix}). 
Now, as in Lemma~\ref{omegaLemma},
$$
\abs{\Dmt \rho_{r_0}} \le
\frac{\norm{\rho'}_{L^\infty}}{r_0^2}\car{|t-s| \le r_0 + \Dt}(t,s),
$$
and therefore
$$
\abs{T_{1,1}} \le C\Dt\frac{r_0 + \Dt}{r_0^2} \int_{\Pi_T}
\abs{S^t_{-\Dt}\Dm\Dp A(u_\Delta)} \dxdt \le C\frac{\Dt}{r_0}\left(1 +
  \frac{\Dt}{r_0}\right).
$$

Next, we consider $T_2$.  By Lemma~\ref{lemma:HolderContDiff},
\begin{align*}
  \abs{T_2} & \le \Dt\int_\nu^\tau \int_\R \abs{\Dmt\Dm\Dp A(u_\Delta)} \dxdt \\
  & \le 2\frac{\Dt}{\Dx} \int_\nu^\tau \norm{\Dmt\Dp A(u_\Delta(\cdot,t))}_{L^1(\R)}\,dt \\
  & \le 2TL\frac{\sqrt{\Dt}}{\Dx}.
\end{align*}
\end{proof}

\begin{proof}[Proof of Theorem~\ref{thm:discreterateExpl}]
  We start out from Lemma~\ref{lemma:MainLemmaTestFuncExpl} with
  $A(\sigma) = \hat{A}(\sigma) + \eta\sigma$, where $\hat{A}$ is the
  original degenerate diffusion function.  By
  Estimate~\ref{est:TimeErrExpl},
  \begin{align*}
    &\left|\Dt \intPi \Dpt \signe(A(u_\Delta)-A(u)) \Dm\Dp
      A(u_\Delta)\test\dX \right| \\ & \qquad \le
    C\frac{\sqrt{\Dt}}{\Dx} + C \frac{\Dt}{r_0} \left(1 +
      \frac{\Dt}{r_0}\right) + C \Dt =: E_4.
  \end{align*}
  Since all the estimates from Section~\ref{sec:Estimates} apply, we
  obtain
  \begin{align*}
    & \intPi
    \abs{u_\Delta-u}\delta^\Dt(t-\tau)\omega_r(x-y)\rho_{r_0}(t-s)\dX
    \\ & \qquad \le \intPi
    \abs{u_\Dx-u}\delta^\Dt(t-\nu)\omega_r(x-y)\rho_{r_0}(t-s) \dX \\
    & \qquad \qquad \qquad + E_1 + E_2 + E_3 + E_4,
  \end{align*}
  where $E_1,E_2, E_3$ are defined respectively in \eqref{eq:E1Def},
  \eqref{eq:E2def}, and \eqref{eq:E3Def}.  Let us make the assumption
  that $\nu = t_m$ and $\tau = t_n$ for some $m,n \in \N$. Then the
  above inequality takes the form
$$
\kappa(t_n) \le \kappa(t_m) +E_1 + E_2 + E_3 + E_4,
$$
where
$$
\kappa(t) := \int_\R \int_{\Pi_T}
\abs{u_\Delta(x,t)-u(y,s)}\omega_r(x-y)\rho_{r_0}(t-s) \,dydsdx.
$$
Applying Lemmas \ref{lem:kappa} and \ref{lem:L1TimeLipDisc}, and
following the reasoning given in the semi-discrete case, we deduce
\begin{align*}%\label{eq:MainIneqDisc}
  &\norm{u_\Delta(\cdot,t_n)-u(\cdot,t_n)}_{L^1(\R)} \\ & \quad \le
  \norm{u^0_\Delta - u^0}_{L^1(\R)} + \left(L_c+L_d\right)t_m \\ &
  \quad \qquad + 2\left(L_c r_0 + \abs{u^0}_{BV(\R)}r\right)
  +C(1+r+\Dx)^2\left(1 + \frac{\Dx}{r}\right)^3\frac{\Dx}{r^2} \\ &
  \quad\qquad + C\frac{\Dx}{r_0} + C\frac{\Dt}{r_0}\left(1 + r_0 +
    \frac{\Dt}{r_0}\right) + C\frac{\sqrt{\Dt}}{\Dx} \\ & \quad \le
  \norm{u^0_\Delta - u^0}_{L^1(\R)} +C\left(\frac{\Dx}{r^2} +
    \frac{\Dx + \Dt}{r_0} + \frac{\sqrt{\Dt}}{\Dx}+ r + r_0\right),
\end{align*}
where $L_d$ is the constant in Lemma~\ref{lem:L1TimeLipDisc} and $L_c$
is the constant from Lemma~\ref{lem:kappa}.
Let $r = r_0, \Dx = r^3$ and $\Dt = r^8$. It follows that
\begin{equation*}
  \norm{u_\Delta(\cdot,t_n)-u(\cdot,t_n)}_{L^1(\R)} 
  \le \norm{u^0_\Delta - u^0}_{L^1(\R)} + C\Dx^{1/3}.
\end{equation*}
Finally, we send $\eta\to 0$ to conclude the proof of the theorem.
\end{proof}

\section{Concluding remarks}\label{sec:final}

The added complexity of convection-diffusion equations versus 
conservation laws \cite{Kuznetsov:1976ys} arises as a 
result of the need to work with an explicit form of the 
parabolic dissipation term. This is reflected in the fact that the rate of 
convergence is lowered to $1/3$ (from $1/2$ for conservation laws) due to 
Estimate~\ref{estimate:DissTermSchemeErr} and Estimate~\ref{est:Re}. 
The optimality of the $\frac13$ rate is an open problem. 
Concerning Section~\ref{sec: error estimate explicit} (explicit schemes), one 
may wonder if it is possible to remove the strengthened CFL 
condition $\Dt \sim \Dx^{8/3}$ (the usual one demands $\Dt \sim \Dx^2$). 
The difficulty is that the parabolic dissipation term is needed to balance the temporal 
error contribution as well as to carry out the doubling-of-the-variables 
argument, and this forces us to impose a stronger relation between $\Dt$ and $\Dx$ in order 
to appropriately control the temporal error contribution. We do not know if the condition $\Dt \sim \Dx^{8/3}$ 
is genuinely needed or is simply an artifact of our method of proof. 
Finally, we are currently investigating the multidimensional case. For the
semi-discrete scheme the main challenge seems to be the adaptation of
Estimate~\ref{est:Re}, or more precisely to produce a multidimensional analogue of \eqref{eq:errRepr}. 
As an additional difficulty, Lemma~\ref{lemma:HolderContDiff} is not 
available in several space dimensions, see \cite{Evje:2000vn}. 
At the moment our multidimensional 
convergence rates are lower than in the one-dimensional case.

\end{document}